\documentclass[journal,onecolumn,12pt]{article}

\usepackage{amssymb}
\usepackage{caption2}
\usepackage{amsmath}
\usepackage{multirow}
\usepackage{amsthm}
\usepackage{graphicx}
\usepackage{CJK}
\usepackage{enumerate}
\usepackage{appendix}
\usepackage{bm}
\usepackage{tikz}
\usetikzlibrary{positioning,chains,fit,shapes,calc}

\newtheorem{theorem}{Theorem}[section]
\newtheorem{lemma}[theorem]{Lemma}
\newtheorem{proposition}[theorem]{Proposition}

\newtheorem{definition}[theorem]{Definition}

\newtheorem{remark}[theorem]{Remark}

\begin{document}
\title{On the lower bound for packing densities of superballs in high dimensions}

\author{Chengfei Xie\thanks{C. Xie is with the School of Mathematical Sciences, Capital Normal University, Beijing 100048, China (email: cfxie@cnu.edu.cn).}
~and Gennian Ge\thanks{Corresponding author. G. Ge is with the School of Mathematical Sciences, Capital Normal University, Beijing 100048, China (e-mail: gnge@zju.edu.cn). The research of G. Ge is supported by the National Key Research and Development Program of China under Grant Nos. 2020YFA0712100  and  2018YFA0704703, National Natural Science Foundation of China under Grant No. 11971325, and Beijing Scholars Program.}}

\maketitle

\begin{abstract}
Define the superball with radius $r$ and center $\bm{0}$ in $\mathbb{R}^n$ to be the set
$$
\left\{\bm{x}\in\mathbb{R}^n:\sum_{j=1}^{m}\left(x_{k_j+1}^2+x_{k_j+2}^2+\cdots+x_{k_{j+1}}^2\right)^{p/2}\leq r^p\right\},0=k_1<k_2<\cdots<k_{m+1}=n,
$$
which is a generalization of $\ell_p$-balls. We give two new proofs for the celebrated result that for $1<p\leq2$, the translative packing density of superballs in $\mathbb{R}^n$ is $\Omega(n/2^n)$. This bound was first obtained by Schmidt, with subsequent constant factor improvement by Rogers and Schmidt, respectively. Our first proof is based on the hard superball model, and the second proof is based on the independence number of a graph. We also investigate the entropy of packings, which measures how plentiful such packings are.
\smallskip
\end{abstract}
\medskip


\smallskip


\section{Introduction}
The sphere packing problem asks the densest packing of nonoverlapping equal-sized balls in $\mathbb{R}^n$. This is an old and difficult problem in discrete geometry. The exact answer is only known in dimensions $1, 2, 3, 8$, and $24$. In dimension $1$, the problem is trivial; in dimension $2$, the problem is nontrivial and can be solved by elementary geometry; in dimension $3$, the problem is known as the Kepler's conjecture and was solved by Hales \cite{MR2179728}; in dimension $8$, the problem was solved by Viazovska \cite{MR3664816} by finding a function that matches the Cohn-Elkies linear programming bound \cite{MR1973059}; in dimension $24$, the problem was solved by Cohn et al$.$ \cite{MR3664817}.

Let $\Delta_{2}(n)$ be the maximum translative packing density of equal-sized Euclidean balls in $\mathbb{R}^n$. The best upper bound of $\Delta_{2}(n)$ in high dimension is due to Kabatjanski\u{\i} and Leven\v{s}te\u{\i}n \cite{MR0514023}: $\Delta_{2}(n)\leq2^{(-0.599\ldots +o(1))n}$. See also Cohn and Zhao \cite{MR3229046} and Sardari and Zargar \cite{2020arXiv200100185S} for constant factor improvement. A lower bound $\Delta_{2}(n)\geq2^{-n}$ is trivial since doubling the radii of balls will cover the whole space. Rogers \cite{MR22863} improved the bound by a factor $n$. Ball \cite{MR1191572} constructed lattice packings of density $2(n-1)2^{-n}\zeta(n)$. The best known lower bound is $(65963+o(1))n2^{-n}$, due to Venkatesh \cite{MR3044452}.

In this paper, we consider the packing density of superballs. Let $k\in\mathbb{N}$ and $p\geq1$ be a real number. Let $\|\cdot\|_{2}$ be the $\ell_2$ norm; that is, for $(x_1, x_2, \ldots, x_k)\in\mathbb{R}^k$, $\|(x_1, x_2, \ldots, x_k)\|_{2}=\sqrt{x_1^2+x_2^2+\cdots+x_k^2}$. Let $\bm{k}=(k_1, k_2, \ldots, k_{m+1})$ be such that $0=k_1<k_2<\cdots<k_{m+1}=n$. Denote the \textit{superball} with radius $r$ and center $\bm{x}$ in $\mathbb{R}^n$ by
\begin{equation*}
B^n_{p, \bm{k}}(\bm{x}, r)=\left\{\bm{y}\in\mathbb{R}^n:\left(\sum_{j=1}^{m}\left\|\left(x_{k_j+1}-y_{k_j+1}, x_{k_j+2}-y_{k_j+2}, \ldots, x_{k_{j+1}}-y_{k_{j+1}}\right)\right\|_{2}^p\right)^{1/p}\leq r\right\},
\end{equation*}
where $\bm{x}=(x_1, x_2, \ldots, x_n)$ and $\bm{y}=(y_1, y_2, \ldots, y_n)$. We simply write $B^n_{p, \bm{k}}(r)=B^n_{p, \bm{k}}(\bm{0}, r)$ if the superball is centered at $\bm{0}$. Here, the role of $\bm{k}$ is to cut the vectors in $\mathbb{R}^n$; that is, $\bm{k}$ cuts $\bm{x}$ and $\bm{y}$ in $\mathbb{R}^n$ into shorter vectors $\bm{x}_j:=\left(x_{k_j+1}, x_{k_j+2}, \ldots, x_{k_{j+1}}\right)$ and $\bm{y}_j:=\left(y_{k_j+1}, y_{k_j+2}, \ldots, y_{k_{j+1}}\right)$, respectively. We first calculate the $\ell_2$-distance between $\bm{x}_j$ and $\bm{y}_j$. And then we calculate the $\ell_p$ norm of the new vector formed by these $\|\bm{x}_j-\bm{y}_j\|_2$.

Throughout the paper, we assume the superballs of our packings have volume $1$ and denote by $r_{p, \bm{k}, n}$ the radius of a superball of volume $1$. Let $\Delta_{p, \bm{k}}(n)$ denote the maximum translative packing density of copies of superballs of volume $1$; that is,
$$
\Delta_{p, \bm{k}}(n)=\limsup_{R\rightarrow\infty}\sup_{\mathcal{P}}\frac{\textrm{vol}(\mathcal{P}\cap B_{p, \bm{k}}^n(R))}{\textrm{vol}(B_{p, \bm{k}}^n(R))},
$$
where $\textrm{vol}(\mathcal{P}\cap B_{p, \bm{k}}^n(R))$ is the volume of $B_{p, \bm{k}}^n(R)$ covered by superballs of volume $1$ with centers in $\mathcal{P}$, and the supremum is over all $\mathcal{P}\subseteq\mathbb{R}^n$ such that the superballs with centers in $\mathcal{P}$ are nonoverlapping.

Consider a special case that $\bm{k}=\bm{k}_n:=(0, 1, 2, \ldots, n)$. In this case,
\begin{equation*}
B^n_{p, \bm{k}_n}(\bm{x}, r)=\left\{\bm{y}\in\mathbb{R}^n:\left(\sum_{j=1}^{n}\left|x_{j}-y_{j}\right|^p\right)^{1/p}\leq r\right\}
\end{equation*}
is the $\ell_p$-ball in $\mathbb{R}^n$. Let $\Delta_{p}(n):=\Delta_{p, \bm{k}_n}(n)$ be the packing density of $\ell_p$-balls. The upper bound for $\Delta_{p}(n)$ was first proved by van der Corput and Schaake \cite{van1936anwendung}, giving $\Delta_{p}(n)\leq\frac{1+n/p}{2^{n/p}}$ if $p\geq2$ and $\Delta_{p}(n)\leq\frac{1+(1-1/p)n}{2^{n/p}}$ if $1\leq p\leq2$. For $p\geq1.494\ldots$, the upper bound was improved exponentially by Sah et al$.$ \cite{MR4064778}. The Minkowski-Hlawka theorem \cite{MR9782} gives a lower bound that $\Delta_{p, \bm{k}}(n)\geq\zeta(n)2^{-n+1}$, where $\zeta(n)$ is the Riemann zeta function. Rush and Sloane \cite{MR908835} improved the Minkowski-Hlawka bound for $\ell_p$-balls and integers $p\geq3$, for instance, $\Delta_{3}(n)\geq2^{-0.8226\ldots n+o(n)}$. Rush \cite{MR1022304} constructed lattice packings with density $2^{-n+o(n)}$ for every convex body which is symmetric through each of the coordinate hyperplanes. Moreover, Elkies et al$.$ \cite{MR1117154} improved the Minkowski-Hlawka bound exponentially for superballs and reals $p>2$, and they also obtained lower bound for the packing density of more general bodies. For the lower bound constructed by error correcting codes, see Rush \cite{MR1147836, MR1307294, MR1375676} and Liu and Xing \cite{MR2384226}.

We focus on the lower bound in the case that $1<p\leq2$. In this case, only subexponential improvements of the Minkowski-Hlawka bound are known. Rogers \cite{MR79045} obtained a lower bound of $\Omega(\sqrt n/2^n)$. Schmidt \cite{MR96638} obtained a lower bound of $\Omega(n/2^n)$. See also Rogers \cite{MR96639} and Schmidt \cite{MR114809} for constant factor improvements. The best result is due to Schmidt \cite{MR146149}, where the constant factor is less than $\log \sqrt2\approx0.346$. We give new proofs for the lower bound $\Omega(n/2^n)$ of the packing density of $\ell_p$-balls.
\begin{theorem}\label{lpqiu}
For every $1<p\leq2$, there exists a constant $c_p\in(0, 2)$ such that
$$
\Delta_{p}(n)\geq(1+o_n(1))\frac{\log(2/c_{p})\cdot n}{2^n}.
$$
\end{theorem}

We also obtain a lower bound for $\Delta_{p, \bm{k}}(n)$.

\begin{theorem}\label{zhuyaojieguo}
For every $p\in(1, 2]$, there exists a constant $c_{p}\in(0, 2)$ such that the following holds. For every $\bm{k}=(k_1, k_2, \ldots, k_{m+1})$ with $k_j\in\mathbb{N}\cup\{0\}$ and $0=k_1<k_2<\cdots<k_{m+1}=n$, we have
$$
\Delta_{p, \bm{k}}(n)\geq(1+o_n(1))\frac{\log(2/c_{p})\cdot n}{2^n}.
$$
\end{theorem}

Observe that Theorem \ref{lpqiu} follows directly from Theorem \ref{zhuyaojieguo} by setting $\bm{k}=(0, 1, 2, \ldots, n)$.

The lower bound in Theorem \ref{zhuyaojieguo} is independent of $\bm{k}$. Note that different $\bm{k}$ will lead to different bodies in the packing. So this result means that, given $1<p\leq2$ and letting $n$ be large, for every $\bm{k}=(k_1, k_2, \ldots, k_{m+1})$ ($k_j\in\mathbb{N}\cup\{0\}$ and $0=k_1<k_2<\cdots<k_{m+1}=n$), the translative packing density of $B^n_{p, \bm{k}}(\bm{0}, r_{p, \bm{k}, n})$ is approximately larger than $\frac{\log(2/c_{p})n}{2^n}$. Throughout the paper, $p$ and $\bm{k}$ appear in the subscripts of most of our symbols and variables since they depend on $p$ and $\bm{k}$. However, we always assume that $p$ is a given number and our bounds are independent of $\bm{k}$ (Theorem \ref{zhuyaojieguo}, Theorem \ref{main}, Theorem \ref{gxiajie} and Theorem \ref{fxiajie}), so $p$ and $\bm{k}$ in the subscripts could be ignored for simplicity.

For $p>2$, we can obtain a lower bound of $\Omega_p(n/2^n)$ as well. But this bound is worse than the bound in \cite{MR1117154}.

We will give two proofs of Theorem \ref{zhuyaojieguo}. In Section \ref{diyizhong}, we give the first proof. The method we use is the hard superball model, which is developed from statistical physics. It was used by Jenssen et al$.$ to prove the lower bound of kissing number \cite{MR3836667} and packing density of Euclidean balls \cite{MR3898718}, where the methods are called hard cap model and hard sphere model, respectively. Recently, Fern{\'a}ndez et al$.$ \cite{2021arXiv211101255G} gave a constant factor improvement of results in \cite{MR3836667} and \cite{MR3898718}. In Section \ref{dierzhong}, we give the second proof, which uses the independence number of a graph. We also use the so-called uniform convexity to overcome the difficulty for the non-Euclideanness.

We also investigate the entropy density and pressure of packings, which we will define in Section \ref{entropy}. This measures how plentiful such packings are.
\section{Uniformly convex spaces}
Throughout the paper, we always assume that $\bm{k}$ is a (finite or infinite) sequence of strictly increasing nonnegative integers. Let $p\geq1$. For $\bm{k}=(k_1, k_2, \ldots, k_{m+1})$ such that $0=k_1<k_2<\cdots<k_{m+1}=n$  and $\bm{x}=(x_1, x_2, \ldots, x_n)\in\mathbb{R}^n$, define
$$
\|\bm{x}\|_{p, \bm{k}, n}=\left(\sum_{j=1}^{m}\left\|\left(x_{k_j+1}, x_{k_j+2}, \ldots, x_{k_{j+1}}\right)\right\|_{2}^p\right)^{1/p}=\left(\sum_{j=1}^{m}\|\bm{x}_j\|_{2}^p\right)^{1/p},
$$
where $\bm{x}_j=(x_{k_j+1}, x_{k_{j}+2}, \ldots, x_{k_{j+1}})$.
For $\bm{k}=(k_1, k_2, \ldots, k_{m+1}, \ldots)$ with $0=k_1<k_2<\cdots<k_{m+1}<\cdots$ (obviously $\lim_{j\rightarrow\infty}k_j=\infty$, since $k_j$'s are nonnegative integers) and $\bm{x}=(x_1, x_2, x_3, \ldots)$, define
$$
\|\bm{x}\|_{p, \bm{k}}=\left(\sum_{j=1}^{\infty}\left\|\left(x_{k_j+1}, x_{k_j+2}, \ldots, x_{k_{j+1}}\right)\right\|_{2}^p\right)^{1/p}=\left(\sum_{j=1}^{\infty}\|\bm{x}_j\|_{2}^p\right)^{1/p},
$$
where $\bm{x}_j=(x_{k_j+1}, x_{k_j+2}, \ldots, x_{k_{j+1}})$. Let $\ell_{p, \bm{k}}$ be the set consisting of all real sequences $\bm{x}$ satisfying $\|\bm{x}\|_{p, \bm{k}}<\infty$.

\begin{proposition}
For every $\bm{k}=(k_1, k_2, \ldots, k_{m+1}, \ldots)$ with $0=k_1<k_2<\cdots<k_{m+1}<\cdots$, and for every $p\geq1$, $\ell_{p, \bm{k}}$ is a normed linear space equipped with the norm $\|\cdot\|_{p, \bm{k}}$.
\end{proposition}
\begin{proof}
For every $\bm{x}=(x_1, x_2, \ldots), \bm{y}=(y_1, y_2, \ldots)\in\ell_{p, \bm{k}}$ and $a\in\mathbb{R}$, let
$$
\bm{x}_j=(x_{k_j+1}, x_{k_j+2}, \ldots, x_{k_{j+1}})
$$
and
$$
\bm{y}_j=(y_{k_j+1}, y_{k_j+2}, \ldots, y_{k_{j+1}}).
$$
We have
\begin{equation}\label{triangle}
\begin{split}
\|\bm{x}+\bm{y}\|_{p, \bm{k}}=&\left(\sum_{j=1}^{\infty}\|\bm{x}_j+\bm{y}_j\|_{2}^p\right)^{1/p}\\
\leq&\left(\sum_{j=1}^{\infty}\left(\|\bm{x}_j\|_{2}+\|\bm{y}_j\|_{2}\right)^p\right)^{1/p}\\
\leq&\left(\sum_{j=1}^{\infty}\|\bm{x}_j\|_{2}^p\right)^{1/p}+\left(\sum_{j=1}^{\infty}\|\bm{y}_j\|_{2}^p\right)^{1/p}\\
=&\|\bm{x}\|_{p, \bm{k}}+\|\bm{y}\|_{p, \bm{k}}<\infty,
\end{split}
\end{equation}
where the first `$\leq$' follows from the triangle inequality of $\|\cdot\|_{2}$ and the second `$\leq$' follows from the Minkowski inequality
\begin{equation}\label{minkowski}
\left(\sum_{j=1}^{\infty}\left(A_j+B_j\right)^s\right)^{1/s}
\leq\left(\sum_{j=1}^{\infty}A_j^s\right)^{1/s}+\left(\sum_{j=1}^{\infty}B_j^s\right)^{1/s}
\end{equation}
for every nonnegative real sequences $(A_j)_{j\in\mathbb{N}}$ and $(B_j)_{j\in\mathbb{N}}$ and $s\geq1$ (if $0<s\leq1$, then inequality (\ref{minkowski}) is in the reverse sense).
And
\begin{equation}\label{linear}
\begin{split}
\|a\bm{x}\|_{p, \bm{k}}=&\left(\sum_{j=1}^{\infty}\|a\bm{x}_j\|_{2}^p\right)^{1/p}\\
=&\left(\sum_{j=1}^{\infty}a^p\|\bm{x}_j\|_{2}^p\right)^{1/p}\\
=&|a|\left(\sum_{j=1}^{\infty}\|\bm{x}_j\|_{2}^p\right)^{1/p}<\infty.
\end{split}
\end{equation}
So $\ell_{p, \bm{k}}$ is a linear space with the usual addition and scalar multiplication.

Clearly, $\|\bm{x}\|_{p, \bm{k}}\geq0$ for every $\bm{x}\in\ell_{p, \bm{k}}$, and if $\|\bm{x}\|_{p, \bm{k}}=0$, then $\bm{x}_j=\bm{0}$ for every $j$, i.e. $\bm{x}=\bm{0}$. We have already proved positive homogeneity (equation (\ref{linear})) and triangle inequality (inequality (\ref{triangle})). We are done.
\end{proof}

For similar reasons, $\|\bm{x}\|_{p, \bm{k}, n}$ is a norm on $\mathbb{R}^n$ as well. For $\bm{x}, \bm{y}\in\mathbb{R}^n$, let $d_{p, \bm{k}, n}(\bm{x}, \bm{y}):=\|\bm{y}-\bm{x}\|_{p, \bm{k}, n}$ be the \textit{$\ell_{p, \bm{k}}$-distance} between $\bm{x}$ and $\bm{y}$ in $\mathbb{R}^n$. For sequences $\bm{x}, \bm{y}\in\ell_{p, \bm{k}}$, let  $d_{p, \bm{k}}(\bm{x}, \bm{y}):=\|\bm{y}-\bm{x}\|_{p, \bm{k}}$ be the \textit{$\ell_{p, \bm{k}}$-distance} between $\bm{x}$ and $\bm{y}$ in $\ell_{p, \bm{k}}$.

\begin{definition}
A normed linear space equipped with norm $\|\cdot\|$ is said to be uniformly convex if to each $\epsilon\in(0, 2]$, there corresponds a $\delta(\epsilon)>0$ such that the conditions $\|x\|=\|y\|=1$ and $\|x-y\|\geq\epsilon$ imply $\left\|\frac{x+y}{2}\right\|\leq1-\delta(\epsilon)$.
\end{definition}

If we write $\bm{k}_\infty=(0, 1, 2, 3, \ldots)$, then the space $\ell_{p, \bm{k}_\infty}$ is the usual $\ell^p$ space. And we have the following theorem.

\begin{theorem}[\cite{MR1501880}]\label{uniformconvex}
For $p>1$, the $\ell_{p, \bm{k}_\infty}$ space is uniformly convex. If $1<p\leq2$ and $\delta_{p, \bm{k}_\infty}(\epsilon)$ is the constant with respect to the uniform convexity of $\ell_{p, \bm{k}_\infty}$ space, then we can choose $\delta_{p, \bm{k}_\infty}(\epsilon)=1-\left(1-\left(\frac{\epsilon}{2}\right)^q\right)^{1/q}$, where $q=p/(p-1)$ is the conjugate index.
\end{theorem}

We have the following generalization of Theorem \ref{uniformconvex}.

\begin{proposition}\label{yizhitu}
For every $\bm{k}=(k_1, k_2, \ldots, k_{m+1}, \ldots)$ with $0=k_1<k_2<\cdots<k_{m+1}<\cdots$, and for every $p>1$, $\ell_{p, \bm{k}}$ is uniformly convex. If $1<p\leq2$ and $\delta_{p}(\epsilon)$ is the constant with respect to the uniform convexity of $\ell_{p, \bm{k}}$ space, then we can choose $\delta_{p}(\epsilon)=1-\left(1-\left(\frac{\epsilon}{2}\right)^q\right)^{1/q}$, where $q=p/(p-1)$ is the conjugate index.
\end{proposition}
\begin{remark}
We can choose $\delta_{p}(\epsilon)$ independent of $\bm{k}$.
\end{remark}
The proof of Proposition \ref{yizhitu} is similar to the proof of uniform convexity of $\ell_{p, \bm{k}_\infty}$ in \cite{MR1501880}, so we first introduce the following lemma, which is a generalization of \cite[Theorem 2]{MR1501880}.
\begin{lemma}
For the space $\ell_{p, \bm{k}}$ defined above, with $p\geq2$, the following inequalities between the norms of two arbitrary elements $\bm{x}$ and $\bm{y}$ of the space are valid (Here $q$ is the conjugate index, $q=p/(p-1)$).
\begin{equation}\label{budengshi3}
2\left(\|\bm{x}\|_{p, \bm{k}}^p+\|\bm{y}\|_{p, \bm{k}}^p\right)\leq\|\bm{x}+\bm{y}\|_{p, \bm{k}}^p+\|\bm{x}-\bm{y}\|_{p, \bm{k}}^p\leq2^{p-1}\left(\|\bm{x}\|_{p, \bm{k}}^p+\|\bm{y}\|_{p, \bm{k}}^p\right);
\end{equation}
\begin{equation}\label{budengshi4}
2\left(\|\bm{x}\|_{p, \bm{k}}^p+\|\bm{y}\|_{p, \bm{k}}^p\right)^{q-1}\leq\|\bm{x}+\bm{y}\|_{p, \bm{k}}^q+\|\bm{x}-\bm{y}\|_{p, \bm{k}}^q;
\end{equation}
\begin{equation}\label{budengshi5}
\|\bm{x}+\bm{y}\|_{p, \bm{k}}^p+\|\bm{x}-\bm{y}\|_{p, \bm{k}}^p\leq2\left(\|\bm{x}\|_{p, \bm{k}}^q+\|\bm{y}\|_{p, \bm{k}}^q\right)^{p-1}.
\end{equation}
For $1<p\leq2$ these inequalities hold in the reverse sense.
\end{lemma}
\begin{proof}
For simplicity, we write $\ell=\ell_{p, \bm{k}}$ and $\|\cdot\|=\|\cdot\|_{p, \bm{k}}$ in the proof. First note that for all values of $p$ the right hand side of inequality (\ref{budengshi3}) is equivalent to the left hand side, while inequality (\ref{budengshi4}) is equivalent to inequality (\ref{budengshi5}); to see this, set $\bm{x}+\bm{y}=\bm{\xi}, \bm{x}-\bm{y}=\bm{\eta}$ and we have
\begin{equation*}
\begin{split}
&\|\bm{x}+\bm{y}\|^p+\|\bm{x}-\bm{y}\|^p\leq2^{p-1}(\|\bm{x}\|^p+\|\bm{y}\|^p)\\
\Leftrightarrow&\|\bm{\xi}\|^p+\|\bm{\eta}\|^p\leq2^{p-1}(\|(\bm{\xi}+\bm{\eta})/2\|^p+\|(\bm{\xi}-\bm{\eta})/2\|^p)\\
\Leftrightarrow&\|\bm{\xi}\|^p+\|\bm{\eta}\|^p\leq2^{-1}(\|\bm{\xi}+\bm{\eta}\|^p+\|\bm{\xi}-\bm{\eta}\|^p)\\
\Leftrightarrow&2(\|\bm{\xi}\|^p+\|\bm{\eta}\|^p)\leq\|\bm{\xi}+\bm{\eta}\|^p+\|\bm{\xi}-\bm{\eta}\|^p,
\end{split}
\end{equation*}
and
\begin{equation*}
\begin{split}
&2(\|\bm{x}\|^p+\|\bm{y}\|^p)^{q-1}\leq\|\bm{x}+\bm{y}\|^q+\|\bm{x}-\bm{y}\|^q\\
\Leftrightarrow&2(\|(\bm{\xi}+\bm{\eta})/2\|^p+\|(\bm{\xi}-\bm{\eta})/2\|^p)^{q-1}\leq\|\bm{\xi}\|^q+\|\bm{\eta}\|^q\\
\Leftrightarrow&2^{1-p(q-1)}(\|\bm{\xi}+\bm{\eta}\|^p+\|\bm{\xi}-\bm{\eta}\|^p)^{q-1}\leq\|\bm{\xi}\|^q+\|\bm{\eta}\|^q\\
\Leftrightarrow&\|\bm{\xi}+\bm{\eta}\|^p+\|\bm{\xi}-\bm{\eta}\|^p\leq2(\|\bm{\xi}\|^q+\|\bm{\eta}\|^q)^{p-1}.
\end{split}
\end{equation*}

We first prove inequality (\ref{budengshi4}) for $1<p\leq2$, i.e.
\begin{equation}\label{budengshi4pie}
2(\|\bm{x}\|^p+\|\bm{y}\|^p)^{q-1}\geq\|\bm{x}+\bm{y}\|^q+\|\bm{x}-\bm{y}\|^q.
\end{equation}
We claim that for $\bm{\alpha}, \bm{\beta}\in\mathbb{R}^k$, we have
\begin{equation}\label{budengshi6}
2\left(\|\bm{\alpha}\|_{2}^p+\|\bm{\beta}\|_{2}^p\right)^{q-1}\geq\|\bm{\alpha}+\bm{\beta}\|_{2}^q+\|\bm{\alpha}-\bm{\beta}\|_{2}^q.
\end{equation}
If one of $\bm{\alpha}$ and $\bm{\beta}$ is $\bm{0}$, the inequality holds trivially. Assume that $\|\bm{\alpha}\|_{2}\geq\|\bm{\beta}\|_{2}>0$. Divide by $\|\bm{\alpha}\|_{2}^q$ in the both sides of inequality (\ref{budengshi6}) and use the homogeneity, reducing inequality (\ref{budengshi6}) to
\begin{equation}\label{budengshi7}
2\left(\left\|\frac{\bm{\alpha}}{\|\bm{\alpha}\|_{2}}\right\|_{2}^p+\left\|\frac{\bm{\beta}}{\|\bm{\alpha}\|_{2}}\right\|_{2}^p\right)^{q-1}\geq\left\|\frac{\bm{\alpha}}{\|\bm{\alpha}\|_{2}}+\frac{\bm{\beta}}{\|\bm{\alpha}\|_{2}}\right\|_{2}^q+\left\|\frac{\bm{\alpha}}{\|\bm{\alpha}\|_{2}}-\frac{\bm{\beta}}{\|\bm{\alpha}\|_{2}}\right\|_{2}^q.
\end{equation}
Let $\bm{a}=\frac{\bm{\alpha}}{\|\bm{\alpha}\|_{2}}$ and $\bm{b}=\frac{\bm{\beta}}{\|\bm{\alpha}\|_{2}}$. Then $\|\bm{a}\|_{2}=1$ and $\|\bm{b}\|_{2}\leq1$, and inequality (\ref{budengshi7}) is equivalent to
\begin{equation*}
2\left(1+\left\|\bm{b}\right\|_{2}^p\right)^{q-1}\geq\left\|\bm{a}+\bm{b}\right\|_{2}^q+\left\|\bm{a}-\bm{b}\right\|_{2}^q,
\end{equation*}
i.e.
\begin{equation*}
2\left(1+\left\|\bm{b}\right\|_{2}^p\right)^{q-1}\geq\left(\left\|\bm{a}\right\|_{2}^2+\left\|\bm{b}\right\|_{2}^2+2\bm{a}\cdot\bm{b}\right)^{q/2}+\left(\left\|\bm{a}\right\|_{2}^2+\left\|\bm{b}\right\|_{2}^2-2\bm{a}\cdot\bm{b}\right)^{q/2},
\end{equation*}
or equivalently,
\begin{equation*}
2\left(1+\left\|\bm{b}\right\|_{2}^p\right)^{q-1}\geq\left(1+\left\|\bm{b}\right\|_{2}^2+2\bm{a}\cdot\bm{b}\right)^{q/2}+\left(1+\left\|\bm{b}\right\|_{2}^2-2\bm{a}\cdot\bm{b}\right)^{q/2},
\end{equation*}
where $\bm{a}\cdot\bm{b}$ is the usual inner product.
Suppose $\left\|\bm{b}\right\|_{2}=b\in[0, 1]$, so
$$
|\bm{a}\cdot\bm{b}|\leq\left\|\bm{a}\right\|_{2}\left\|\bm{b}\right\|_{2}=b.
$$
Without loss of generality assume that $\bm{a}\cdot\bm{b}\in[0, b]$. Consider the function
$$
g(x)=(1+b^2+2x)^{q/2}+(1+b^2-2x)^{q/2}, x\in[0, b].
$$
We have $g'(x)=q(1+b^2+2x)^{q/2-1}-q(1+b^2-2x)^{q/2-1}\geq0$ for $x\in[0, b]$, since $q\geq2$. Thus $\max g(x)=g(b)=(1+b^2+2b)^{q/2}+(1+b^2-2b)^{q/2}=(b+1)^q+(1-b)^q$ and it suffices to prove
\begin{equation}\label{budengshi7pie}
2\left(1+b^p\right)^{q-1}\geq\left(1+b\right)^{q}+\left(1-b\right)^{q}.
\end{equation}
Inequality (\ref{budengshi7pie}) has been proved in \cite{MR1501880} (see the proof of \cite[Theorem 2]{MR1501880}). So inequality (\ref{budengshi6}) is proved.

Now we are able to prove inequality (\ref{budengshi4pie}). Let $\bm{x}=(\bm{x}_1, \bm{x}_2, \ldots)=(x_1, x_2, \ldots), \bm{y}=(\bm{y}_1, \bm{y}_2, \ldots)=(y_1, y_2, \ldots)\in\ell$, where
$$
\bm{x}_j=(x_{k_j+1}, x_{k_j+2}, \ldots, x_{k_{j+1}})
$$
and
$$
\bm{y}_j=(y_{k_j+1}, y_{k_j+2}, \ldots, y_{k_{j+1}}).
$$
Inequality (\ref{budengshi4pie}) states that
\begin{equation}\label{budengshi8}
2\left(\sum_{j=1}^{\infty}\left(\|\bm{x}_j\|_{2}^p+\|\bm{y}_j\|_{2}^p\right)\right)^{q-1}\geq\left(\sum_{j=1}^{\infty}\|\bm{x}_j+\bm{y}_j\|_{2}^p\right)^{q/p}+\left(\sum_{j=1}^{\infty}\|\bm{x}_j-\bm{y}_j\|_{2}^p\right)^{q/p}.
\end{equation}
Using Minkowski inequality (\ref{minkowski}) in the reverse sense with $A_j=\|\bm{x}_j+\bm{y}_j\|_{2}^q, B_j=\|\bm{x}_j-\bm{y}_j\|_{2}^q$ and $s=p/q\leq1$, we infer that the right hand side of inequality (\ref{budengshi8}) is
$$
\left(\sum_{j=1}^{\infty}A_j^s\right)^{1/s}+\left(\sum_{j=1}^{\infty}B_j^s\right)^{1/s}\leq\left(\sum_{j=1}^{\infty}\left(A_j+B_j\right)^s\right)^{1/s}=\left(\sum_{j=1}^{\infty}\left(\|\bm{x}_j+\bm{y}_j\|_{2}^q+\|\bm{x}_j-\bm{y}_j\|_{2}^q\right)^{p/q}\right)^{q/p},
$$
which by inequality (\ref{budengshi6}) is
$$
\leq\left(\sum_{j=1}^{\infty}\left(2\left(\|\bm{x}_j\|_{2}^p+\|\bm{y}_j\|_{2}^p\right)^{q-1}\right)^{p/q}\right)^{q/p}=2\left(\sum_{j=1}^{\infty}\left(\|\bm{x}_j\|_{2}^p+\|\bm{y}_j\|_{2}^p\right)\right)^{q/p}.
$$
Since $q/p=q-1$, this is our result. It completes the proof of inequality (\ref{budengshi4}) for $1<p\leq2$.

Now we prove inequality (\ref{budengshi4}) for $p\geq2$, and we need to prove inequality (\ref{budengshi8}) in the reverse sense, i.e.
\begin{equation}\label{budengshi8pie}
2\left(\sum_{j=1}^{\infty}\left(\|\bm{x}_j\|_{2}^p+\|\bm{y}_j\|_{2}^p\right)\right)^{q-1}\leq\left(\sum_{j=1}^{\infty}\|\bm{x}_j+\bm{y}_j\|_{2}^p\right)^{q/p}+\left(\sum_{j=1}^{\infty}\|\bm{x}_j-\bm{y}_j\|_{2}^p\right)^{q/p}.
\end{equation}
Letting $A_j, B_j$ and $s$ have the same values as above, and again applying Minkowski inequality (\ref{minkowski}), we conclude that the right hand side of inequality (\ref{budengshi8pie}) is
\begin{equation}\label{budengshi9}
\left(\sum_{j=1}^{\infty}A_j^s\right)^{1/s}+\left(\sum_{j=1}^{\infty}B_j^s\right)^{1/s}\geq\left(\sum_{j=1}^{\infty}\left(A_j+B_j\right)^s\right)^{1/s}=\left(\sum_{j=1}^{\infty}\left(\|\bm{x}_j+\bm{y}_j\|_{2}^q+\|\bm{x}_j-\bm{y}_j\|_{2}^q\right)^{p/q}\right)^{q/p}.
\end{equation}
Recall that inequality (\ref{budengshi6}) holds for $1<p\leq2$, and it is equivalent to
$$
\|\bm{\alpha}+\bm{\beta}\|_{2}^p+\|\bm{\alpha}-\bm{\beta}\|_{2}^p\geq2\left(\|\bm{\alpha}\|_{2}^q+\|\bm{\beta}\|_{2}^q\right)^{p-1}.
$$
After exchanging the role of $p$ and $q$, we have
$$
\|\bm{\alpha}+\bm{\beta}\|_{2}^q+\|\bm{\alpha}-\bm{\beta}\|_{2}^q\geq2\left(\|\bm{\alpha}\|_{2}^p+\|\bm{\beta}\|_{2}^p\right)^{q-1},
$$
for $p\geq2$.
Thus inequality ($\ref{budengshi9}$) is
$$
\geq\left(\sum_{j=1}^{\infty}\left(2\left(\|\bm{x}_j\|_{2}^p+\|\bm{y}_j\|_{2}^p\right)^{q-1}\right)^{p/q}\right)^{q/p}=2\left(\sum_{j=1}^{\infty}\left(\|\bm{x}_j\|_{2}^p+\|\bm{y}_j\|_{2}^p\right)\right)^{q-1}.
$$
Therefore, we complete the proof of inequality (\ref{budengshi4}) for $p\geq2$.

Finally we prove inequality (\ref{budengshi3}). Let $p\geq2$ and consider the right hand inequality. This is simply implied by inequality (\ref{budengshi5}): For $x, y\geq0$, we have
\begin{equation}\label{budengshi10}
2(x^q+y^q)^{p-1}\leq2^{p-1}(x^p+y^p).
\end{equation}
Inequality (\ref{budengshi10}) has been proved in \cite{MR1501880} (see the proof of \cite[Theorem 2]{MR1501880}). For $1<p\leq2$, it follows from inequality (\ref{budengshi10}) that
\begin{equation}\label{budengshi10pie}
2(x^q+y^q)^{p-1}\geq2^{p-1}(x^p+y^p).
\end{equation}
The right hand side of inequality (\ref{budengshi3}) (in the reverse sense) follows from inequality (\ref{budengshi5}) (in the reverse sense) and inequality (\ref{budengshi10pie}).
\end{proof}

Now we are able to prove Proposition \ref{yizhitu}.

\begin{proof}[Proof of Proposition \ref{yizhitu}]
For $p\geq2$, letting $\|\bm{x}\|_{p, \bm{k}}=\|\bm{y}\|_{p, \bm{k}}=1$ in inequality (\ref{budengshi3}), we have
$$
\|\bm{x}+\bm{y}\|_{p, \bm{k}}^p+\|\bm{x}-\bm{y}\|_{p, \bm{k}}^p\leq2^{p}.
$$
If $\|\bm{x}-\bm{y}\|_{p, \bm{k}}\geq\epsilon$, then
$$
\left\|\frac{\bm{x}+\bm{y}}{2}\right\|_{p, \bm{k}}\leq(1-(\epsilon/2)^p)^{1/p}.
$$
So we can choose $\delta(\epsilon)=1-(1-(\epsilon/2)^p)^{1/p}$. Similarly, for $1<p\leq2$, we can choose $\delta(\epsilon)=1-(1-(\epsilon/2)^q)^{1/q}$ by inequality (\ref{budengshi4}).
\end{proof}

\section{The hard superball model}\label{diyizhong}
Given $p>1, n\in\mathbb{N}$ and $\bm{k}=(k_1, k_2, \ldots, k_{m+1})$ with $0=k_1<k_2<\cdots<k_{m+1}=n$, consider the translative packing of $B_{p, \bm{k}}^n(\bm{0}, r_{p, \bm{k}, n})$. For a bounded, measurable subset $S\subseteq\mathbb{R}^n$, let
$$
P_{t, p, \bm{k}}(S)=\{\{\bm{x}_1, \bm{x}_2, \ldots, \bm{x}_t\}\subseteq S:d_{p, \bm{k}, n}(\bm{x}_i, \bm{x}_j)\geq2r_{p, \bm{k}, n}\text{ for every }i\neq j\}.
$$
be the family consisting of unordered $t$-tuples of points from $S$ that can form a packing.

The \textit{canonical hard superball model} on $S$ with $t$ centers is simply a uniformly random $t$-tuple $X_{t, p, \bm{k}}\in P_{t, p, \bm{k}}(S)$. The partition function of the canonical hard superball model on $S$ is the function
\begin{equation}\label{defzhat}
\hat{Z}_{S, p, \bm{k}}(t)=\frac{1}{t!}\int_{S^t}\textbf{1}_{\mathcal{D}_{p, \bm{k}}(\bm{x}_1, \bm{x}_2, \ldots, \bm{x}_t)}d\bm{x}_1d\bm{x}_2\cdots d\bm{x}_t,
\end{equation}
where for $\bm{x}_1, \bm{x}_2, \ldots, \bm{x}_t\in\mathbb{R}^n$, the expression $\mathcal{D}_{p, \bm{k}}(\bm{x}_1, \bm{x}_2, \ldots, \bm{x}_t)$ denotes the event that $d_{p, \bm{k}, n}(\bm{x}_i, \bm{x}_j)\geq2r_{p, \bm{k}, n}$ for every $i\neq j$.

The \textit{grand canonical hard superball model} on $S$ at fugacity $\lambda$ is a random set $X$ of unordered points, with $X$ distributed according to a Poisson point process of intensity $\lambda$ conditioned on the event that $d_{p, \bm{k}, n}(\bm{x}, \bm{y})\geq2r_{p, \bm{k}, n}$ for all distinct $\bm{x}, \bm{y}\in X$. In other words, we first choose $t$ from $\{0, 1, 2, \ldots\}$ with probability $\propto$ $\lambda^t\hat{Z}_{S, p, \bm{k}}(t)$, then we independently, uniformly choose $X$ from $P_{t, p, \bm{k}}(S)$. The partition function of the grand canonical hard superball model on $S$ is
\begin{equation}
Z_{S, p, \bm{k}}(\lambda)=\sum_{t=0}^{\infty}\lambda^t\hat{Z}_{S, p, \bm{k}}(t),
\end{equation}
where we take $\hat{Z}_{S, p, \bm{k}}(0)=1$. Note that if $S$ is bounded then $Z_{S, p, \bm{k}}(\lambda)$ is a polynomial in $\lambda$.

The \textit{expected packing density}, $\alpha_{S, p, \bm{k}}(\lambda)$, of the hard superball model is simply the expected number of centers in $S$ normalized by the volume of $S$; that is,
$$
\alpha_{S, p, \bm{k}}(\lambda)=\frac{\mathbb{E}_{S, p, \bm{k}, \lambda}|X|}{\text{vol}(S)}.
$$
Here and in what follows the notation $\mathbb{P}_{S, p, \bm{k}, \lambda}$ and $\mathbb{E}_{S, p, \bm{k}, \lambda}$ indicates probabilities and expectations with respect to the grand canonical hard superball model on a region $S$ at fugacity $\lambda$.

The expected packing density can be expressed as the derivative of the normalized log partition function.
\begin{equation}\label{dengshi4}
\begin{split}
\alpha_{S, p, \bm{k}}(\lambda)&=\frac{1}{\text{vol}(S)}\sum_{t=1}^\infty t\cdot\mathbb{P}_{S, p, \bm{k}, \lambda}(|X|=t)\\
&=\frac{1}{\text{vol}(S)}\sum_{t=1}^\infty\frac{t\cdot\lambda^t\hat{Z}_{S, p, \bm{k}}(t)}{Z_{S, p, \bm{k}}(\lambda)}\\
&=\frac{1}{\text{vol}(S)}\frac{\lambda\cdot (Z_{S, p, \bm{k}}(\lambda))'}{Z_{S, p, \bm{k}}(\lambda)}\\
&=\frac{\lambda}{\text{vol}(S)}(\log Z_{S, p, \bm{k}}(\lambda))'.
\end{split}
\end{equation}
Here and in what follows $\log x$ always denotes the natural logarithm of $x$.

\begin{lemma}\label{xiajie}
The asymptotic expected packing density of $B^n_{p, \bm{k}}(R)\subseteq\mathbb{R}^n$ is a lower bound on the maximum superball packing density. That is, for any $\lambda>0$,
$$
\Delta_{p, \bm{k}}(n)\geq\limsup_{R\rightarrow\infty}\alpha_{B^n_{p, \bm{k}}(R), p, \bm{k}}(\lambda).
$$
\end{lemma}
\begin{proof}
By the definition of $\Delta_{p, \bm{k}}(n)$,
$$
\Delta_{p, \bm{k}}(n)=\limsup_{R\rightarrow\infty}\sup_{X\in\mathcal{P}}\frac{|X|}{(R/r_{p, \bm{k}, n})^n},
$$
where $\mathcal{P}$ is the family consisting of all packings of $B^n_{p, \bm{k}}(R)$ by superballs of radius $r=r_{p, \bm{k}, n}$, and $(R/r_{p, \bm{k}, n})^n$ is the volume of $B^n_{p, \bm{k}}(R)$. Note that in our model, some centers may be near the boundary of $B^n_{p, \bm{k}}(R)$, which should be deleted from the packing. However, we can enlarge the radius of the superball to make it be a packing. In other words, if $X$ is chosen from the model on $B^n_{p, \bm{k}}(R)$, then $X$ is a packing of $B^n_{p, \bm{k}}(R+100)$. So
\begin{equation*}
\begin{split}
\Delta_{p, \bm{k}}(n)&\geq\limsup_{R\rightarrow\infty}\frac{\mathbb{E}_{B^n_{p, \bm{k}}(R), p, \bm{k}, \lambda}(|X|)}{((R+100)/r_{p, \bm{k}, n})^n}\\
&=\limsup_{R\rightarrow\infty}\frac{\mathbb{E}_{B^n_{p, \bm{k}}(R), p, \bm{k}, \lambda}(|X|)}{((R+100)/r_{p, \bm{k}, n})^n}\cdot\frac{((R+100)/r_{p, \bm{k}, n})^n}{(R/r_{p, \bm{k}, n})^n}\\
&=\limsup_{R\rightarrow\infty}\frac{\mathbb{E}_{B^n_{p, \bm{k}}(R), p, \bm{k}, \lambda}(|X|)}{(R/r_{p, \bm{k}, n})^n}\\
&=\limsup_{R\rightarrow\infty}\alpha_{B^n_{p, \bm{k}}(R), p, \bm{k}}(\lambda).
\end{split}
\end{equation*}
\end{proof}

\begin{theorem}\label{main}
For every $p\in(1, 2]$, there exists a constant $c_p\in(0, 2)$ such that the following holds. Let $S\subseteq\mathbb{R}^n$ be bounded, measurable, and of positive volume. Then for every $\bm{k}=(k_1, k_2, \ldots, k_{m+1})$ such that $0=k_1<k_2<\cdots<k_{m+1}=n$ and every $\lambda\geq n^{-1}c_p^{-n}$, we have
$$
\alpha_{S, p, \bm{k}}(\lambda)\geq(1+o_n(1))\frac{\log(2/c_p)\cdot n}{2^n}.
$$
\end{theorem}

\begin{remark}
The constant $c_p$ in Theorem \ref{main} is independent of $n$ and $\bm{k}$.
\end{remark}

Theorem \ref{zhuyaojieguo} follows immediately from Lemma \ref{xiajie} and Theorem \ref{main}.

\begin{lemma}\label{dizeng}
Let $S\subseteq\mathbb{R}^n$ be bounded, measurable, and of positive volume. Then the expected packing density $\alpha_{S, p, \bm{k}}(\lambda)$ is strictly increasing in $\lambda$.
\end{lemma}
\begin{proof}
We use equation (\ref{dengshi4}) and calculate
\begin{equation*}
\begin{split}
\lambda\cdot{\rm vol}(S)\cdot\alpha_{S, p, \bm{k}}'(\lambda)&=\lambda\cdot\left(\frac{\lambda\cdot (Z_{S, p, \bm{k}}(\lambda))'}{Z_{S, p, \bm{k}}(\lambda)}\right)'\\
&=\lambda\cdot\left(\frac{(Z_{S, p, \bm{k}}(\lambda))'+\lambda\cdot (Z_{S, p, \bm{k}}(\lambda))''}{Z_{S, p, \bm{k}}(\lambda)}-\frac{\lambda\cdot \left[(Z_{S, p, \bm{k}}(\lambda))'\right]^2}{\left(Z_{S, p, \bm{k}}(\lambda)\right)^2}\right).
\end{split}
\end{equation*}
Since
$$
\lambda\cdot\frac{(Z_{S, p, \bm{k}}(\lambda))'}{Z_{S, p, \bm{k}}(\lambda)}={\rm vol}(S)\cdot\alpha_{S, p, \bm{k}}=\mathbb{E}_{S, p, \bm{k}, \lambda}|X|
$$
and
$$
\lambda\cdot\frac{\lambda\cdot (Z_{S, p, \bm{k}}(\lambda))''}{Z_{S, p, \bm{k}}(\lambda)}=\lambda^2\cdot\frac{\sum_{t=2}^\infty t(t-1)\lambda^{t-2}\hat{Z}_{S, p, \bm{k}}(t)}{Z_{S, p, \bm{k}}(\lambda)}=\sum_{t=2}^\infty t(t-1)\mathbb{P}_{S, p, \bm{k}, \lambda}(|X|=t)=\mathbb{E}_{S, p, \bm{k}, \lambda}[|X|(|X|-1)],
$$
it follows that
\begin{equation}\label{fangcha}
\lambda\cdot{\rm vol}(S)\cdot\alpha_{S, p, \bm{k}}'(\lambda)=\mathbb{E}_{S, p, \bm{k}, \lambda}|X|+\mathbb{E}_{S, p, \bm{k}, \lambda}[|X|(|X|-1)]-\left(\mathbb{E}_{S, p, \bm{k}, \lambda}|X|\right)^2={\rm Var}(|X|)>0.
\end{equation}
Thus $\alpha_{S, p, \bm{k}}(\lambda)$ is strictly increasing.
\end{proof}

Let $\text{FV}_{S, p, \bm{k}}(\lambda)$ denote the expected \textit{free volume} of the hard superball model; that is, the expected fraction of the volume of $S$ containing points that are at distance at least $2r_{p, \bm{k}, n}$ from the nearest center. That is,
$$
\text{FV}_{S, p, \bm{k}}(\lambda)=\frac{\mathbb{E}_{S, p, \bm{k}, \lambda}\left[\text{vol}(\{\bm{y}\in S:d_{p, \bm{k}, n}(\bm{y}, \bm{x})\geq2r_{p, \bm{k}, n}\text{ for every }\bm{x}\in X\})\right]}{\text{vol}(S)}.
$$

\begin{lemma}\label{fvs}
Let $S\subseteq\mathbb{R}^n$ be bounded, measurable, and of positive volume. Then
$$
\alpha_{S, p, \bm{k}}(\lambda)=\lambda\cdot{\rm FV}_{S, p, \bm{k}}(\lambda).
$$
\end{lemma}
\begin{proof}
We use the definitions of $\alpha_{S, p, \bm{k}}(\lambda)$ and ${\rm FV}_{S, p, \bm{k}}(\lambda)$.
\begin{equation*}
\begin{split}
\alpha_{S, p, \bm{k}}(\lambda)&=\frac{\mathbb{E}_{S, p, \bm{k}, \lambda}|X|}{\text{vol}(S)}\\
&=\frac{1}{\text{vol}(S)}\sum_{t=0}^\infty (t+1)\cdot\mathbb{P}_{S, p, \bm{k}, \lambda}(|X|=t+1)\\
&=\frac{1}{\text{vol}(S)Z_{S, p, \bm{k}}(\lambda)}\sum_{t=0}^\infty (t+1)\cdot\int_{S^{t+1}}\frac{\lambda^{t+1}}{(t+1)!}\textbf{1}_{\mathcal{D}_{p, \bm{k}}(\bm{x}_0, \bm{x}_1, \bm{x}_2, \ldots, \bm{x}_t)}d\bm{x}_0d\bm{x}_1\cdots d\bm{x}_t\\
&=\frac{\lambda}{\text{vol}(S)Z_{S, p, \bm{k}}(\lambda)}\sum_{t=0}^\infty\int_{S^{t+1}}\frac{\lambda^{t}}{t!}\textbf{1}_{\mathcal{D}_{p, \bm{k}}(\bm{x}_0, \bm{x}_1, \bm{x}_2, \ldots, \bm{x}_t)}d\bm{x}_0d\bm{x}_1\cdots d\bm{x}_t\\
&=\frac{\lambda}{\text{vol}(S)Z_{S, p, \bm{k}}(\lambda)}\sum_{t=0}^\infty\int_{S^{t}}\frac{\lambda^{t}}{t!}\left(\int_{S}\textbf{1}_{\mathcal{D}_{p, \bm{k}}(\bm{x}_0, \bm{x}_1, \bm{x}_2, \ldots, \bm{x}_t)} d\bm{x}_0\right)d\bm{x}_1\cdots d\bm{x}_t.
\end{split}
\end{equation*}
Let $Y=\text{vol}(\{\bm{y}\in S:d_{p, \bm{k}, n}(\bm{y}, \bm{x})\geq2r_{p, \bm{k}, n}\text{ for every }\bm{x}\in X\})$. Then
$$
\mathbb{E}_{S, p, \bm{k}, \lambda}\left(Y|X=\{\bm{x}_1, \bm{x}_2, \ldots, \bm{x}_t\}\right)=\int_{S}\textbf{1}_{\mathcal{D}_{p, \bm{k}}(\bm{x}_0, \bm{x}_1, \bm{x}_2, \ldots, \bm{x}_t)} d\bm{x}_0.
$$
So
\begin{equation*}
\begin{split}
&\lambda\cdot\text{FV}_{S, p, \bm{k}}(\lambda)\\
=&\frac{\lambda}{\text{vol}(S)}\cdot\mathbb{E}_{S, p, \bm{k}, \lambda}\left(Y\right)\\
=&\frac{\lambda}{\text{vol}(S)}\cdot\mathbb{E}_{S, p, \bm{k}, \lambda}\left[\mathbb{E}_{S, p, \bm{k}, \lambda}\left(Y|X=\{\bm{x}_1, \bm{x}_2, \ldots, \bm{x}_t\}\right)\right]\\
=&\frac{\lambda}{\text{vol}(S)}\sum_{t=0}^\infty\frac{\lambda^{t}}{Z_{S, p, \bm{k}}(\lambda)}\frac{1}{t!}\int_{S^{t}}\left(\int_{S}\textbf{1}_{\mathcal{D}_{p, \bm{k}}(\bm{x}_0, \bm{x}_1, \bm{x}_2, \ldots, \bm{x}_t)} d\bm{x}_0\right)\textbf{1}_{\mathcal{D}_{p, \bm{k}}(\bm{x}_1, \bm{x}_2, \ldots, \bm{x}_t)}d\bm{x}_1\cdots d\bm{x}_t\\
=&\frac{\lambda}{\text{vol}(S)Z_{S, p, \bm{k}}(\lambda)}\sum_{t=0}^\infty\int_{S^{t}}\frac{\lambda^{t}}{t!}\left(\int_{S}\textbf{1}_{\mathcal{D}_{p, \bm{k}}(\bm{x}_0, \bm{x}_1, \bm{x}_2, \ldots, \bm{x}_t)} d\bm{x}_0\right)d\bm{x}_1\cdots d\bm{x}_t\\
=&\alpha_{S, p, \bm{k}}(\lambda),
\end{split}
\end{equation*}
where for $t=0$, $\textbf{1}_{\mathcal{D}_{p, \bm{k}}(\bm{x}_1, \bm{x}_2, \ldots, \bm{x}_t)}\equiv1$ (since $\hat{Z}_{S, p, \bm{k}}(0)=1$).
\end{proof}

Now consider the following two-part experiment: sample a configuration of centers $X$ from the hard superball model on $S$ at fugacity $\lambda$ and independently choose a point $\bm{v}$ from $S$. We define the random set
$$
\text{T}=\{\bm{x}\in B^n_{p, \bm{k}}(\bm{v}, 2r_{p, \bm{k}, n})\cap S:d_{p, \bm{k}, n}(\bm{x}, \bm{y})\geq2r_{p, \bm{k}, n}\text{ for every }\bm{y}\in X\cap B^n_{p, \bm{k}}(\bm{v}, 2r_{p, \bm{k}, n})^c\}.
$$
That is, $\text{T}$ is the set of all points of $S$ in $B^n_{p, \bm{k}}(\bm{v}, 2r_{p, \bm{k}, n})$ that are not blocked from being a center by a center outside of $B^n_{p, \bm{k}}(\bm{v}, 2r_{p, \bm{k}, n})$.

\begin{lemma}\label{xiajie2}
Let $S\subseteq\mathbb{R}^n$ be bounded, measurable, and of positive volume. Then
\begin{equation}\label{qwer1}
\alpha_{S, p, \bm{k}}(\lambda)=\lambda\cdot\mathbb{E}\left[\frac{1}{Z_{{\rm T}, p, \bm{k}}(\lambda)}\right]
\end{equation}
and
\begin{equation}
\alpha_{S, p, \bm{k}}(\lambda)\geq2^{-n}\cdot\mathbb{E}\left[\frac{\lambda\cdot(Z_{{\rm T}, p, \bm{k}}(\lambda))'}{Z_{{\rm T}, p, \bm{k}}(\lambda)}\right],
\end{equation}
where both expectations are with respect to the random set {\rm T} generated by the two-part experiment defined above.
\end{lemma}
\begin{proof}
Using Lemma \ref{fvs}, we have
\begin{equation*}
\begin{split}
\alpha_{S, p, \bm{k}}(\lambda)&=\lambda\cdot{\rm FV}_{S, p, \bm{k}}(\lambda)\\
&=\frac{\lambda}{{\rm vol}(S)}\cdot\int_S\mathbb{P}[d_{p, \bm{k}, n}(\bm{x}, \bm{y})\geq2r_{p, \bm{k}, n}, \forall\bm{x}\in X]d\bm{y}\\
&=\lambda\cdot\mathbb{E}(\mathbf{1}_{{\rm T}\cap X=\emptyset})\\
&=\lambda\cdot\mathbb{E}\left[\frac{1}{Z_{{\rm T}, p, \bm{k}}(\lambda)}\right],
\end{split}
\end{equation*}
where the last equation uses the spatial Markov property of the hard superball model: conditioned on $X\cap B^n_{p, \bm{k}}(\bm{v}, 2r_{p, \bm{k}, n})^c$, the distribution of $X\cap B^n_{p, \bm{k}}(\bm{v}, 2r_{p, \bm{k}, n})$ is exactly that of the hard superball model on the set ${\rm T}$.

For every $\bm{x}\in S$, $\mathbb{P}(\bm{x}\in B^n_{p, \bm{k}}(\bm{v}, 2r_{p, \bm{k}, n}))={\rm vol}(B^n_{p, \bm{k}}(\bm{v}, 2r_{p, \bm{k}, n})\cap S)/{\rm vol}(S)\leq2^n/{\rm vol}(S)$.
So
\begin{equation*}
\begin{split}
\alpha_{S, p, \bm{k}}(\lambda)&=\frac{\mathbb{E}_{S, p, \bm{k}, \lambda}|X|}{\text{vol}(S)}\\
&\geq2^{-n}\mathbb{E}(|X\cap B^n_{p, \bm{k}}(\bm{v}, 2r_{p, \bm{k}, n})|)\\
&=2^{-n}\mathbb{E}(\alpha_{{\rm T}, p, \bm{k}}(\lambda)\cdot{\rm vol(T)})\\
&=2^{-n}\cdot\mathbb{E}\left[\frac{\lambda\cdot(Z_{{\rm T}, p, \bm{k}}(\lambda))'}{Z_{{\rm T}, p, \bm{k}}(\lambda)}\right].
\end{split}
\end{equation*}

\end{proof}

\begin{lemma}\label{xiajie3}
Let $S\subseteq\mathbb{R}^n$ be bounded and measurable. Then
\begin{equation}\label{dengshi8}
\log Z_{S, p, \bm{k}}(\lambda)\leq\lambda\cdot{\rm vol}(S)
\end{equation}
and if in addition $S$ is of positive volume, then
\begin{equation}\label{dengshi3}
\alpha_{S, p, \bm{k}}(\lambda)\geq\lambda\cdot e^{-\lambda\cdot\mathbb{E}\left[\rm vol(T)\right]}.
\end{equation}
\end{lemma}
\begin{proof}
By the definition of $Z_{S, p, \bm{k}}(\lambda)$, we have
\begin{equation*}
\begin{split}
Z_{S, p, \bm{k}}(\lambda)&=\sum_{t=0}^{\infty}\lambda^t\hat{Z}_{S, p, \bm{k}}(t)\\
&=\sum_{t=0}^{\infty}\frac{\lambda^t}{t!}\int_{S^t}\textbf{1}_{\mathcal{D}_{p, \bm{k}}(\bm{x}_1, \bm{x}_2, \ldots, \bm{x}_t)}d\bm{x}_1d\bm{x}_2\cdots d\bm{x}_t\\
&\leq\sum_{t=0}^{\infty}\frac{\lambda^t}{t!}\int_{S^t}d\bm{x}_1d\bm{x}_2\cdots d\bm{x}_t\\
&=\sum_{t=0}^{\infty}\frac{\lambda^t}{t!}({\rm vol}(S))^t\\
&=e^{\lambda\cdot{\rm vol}(S)}.
\end{split}
\end{equation*}
Take logarithm and we obtain inequality (\ref{dengshi8}).

For inequality (\ref{dengshi3}), we use equation (\ref{qwer1}) and inequality (\ref{dengshi8}). So
\begin{equation*}
\begin{split}
\alpha_{S, p, \bm{k}}(\lambda)&=\lambda\cdot\mathbb{E}\left[\frac{1}{Z_{{\rm T}, p, \bm{k}}(\lambda)}\right]\\
&\geq\lambda\cdot\mathbb{E}\left[e^{-\lambda\cdot{\rm vol(T)}}\right]\\
&\geq\lambda\cdot e^{-\lambda\cdot\mathbb{E}\left[{\rm vol(T)}\right]},
\end{split}
\end{equation*}
where the last inequality follows from Jensen's inequality.
\end{proof}

\begin{remark}
In \cite{MR3898718}, Lemmas \ref{dizeng}-\ref{xiajie3} are with respect to the packing of Euclidean balls. Indeed, these lemmas are valid for the packing of superballs defined above as well.
\end{remark}

Consider the function
$$
h(x)=\left(\frac{x}{4}+\frac{1}{2}-\frac{1}{x}\right)^q+\left(\frac{x+2}{4}\right)^q-1.
$$
Note that $h(x)$ is continuous for $x\geq1.5$ and $h(2)=1/2^q>0$, so there exists $x_p\in(1.5, 2)$ such that $h(x)>1/3^q$ for every $x\in[x_p, 2]$. We will use $x_p$ in the proof of the following lemma.

\begin{lemma}\label{zhongyao}
For every $p\in(1, 2]$, there exists a constant $c_p\in(0, 2)$ such that the following holds. For every $n\in\mathbb{N}$ and $\bm{k}=(k_1, k_2, \ldots, k_{m+1})$ with $0=k_1<k_2<\cdots<k_{m+1}=n$, let $S\subseteq B^n_{p, \bm{k}}(2r_{p, \bm{k}, n})$ be measurable. Then
\begin{equation}\label{xiaoyucp}
\mathbb{E}\left[{\rm vol}(B^n_{p, \bm{k}}(\bm{u}, 2r_{p, \bm{k}, n})\cap S)\right]\leq2\cdot c_p^n,
\end{equation}
where $\bm{u}$ is a uniformly chosen point in $S$. In particular,
\begin{equation}\label{dengshi2}
\alpha_{S, p, \bm{k}}\geq\lambda\cdot e^{-\lambda\cdot2\cdot c_p^n}.
\end{equation}
\end{lemma}
\begin{remark}
The constant $c_p$ in Theorem \ref{main} will be chosen as the same as the constant $c_p$ here. So we use the same symbol.
\end{remark}
\begin{proof}
Clearly, we may assume that $S$ has positive volume. We have
\begin{equation*}
\begin{split}
\mathbb{E}\left[{\rm vol}(B^n_{p, \bm{k}}(\bm{u}, 2r_{p, \bm{k}, n})\cap S)\right]&=\frac{1}{{\rm vol}(S)}\int_S\int_S\mathbf{1}_{d_{p, \bm{k}, n}(\bm{u}, \bm{v})\leq2r_{p, \bm{k}, n}}d\bm{v}d\bm{u}\\
&=\frac{2}{{\rm vol}(S)}\int_S\int_S\mathbf{1}_{d_{p, \bm{k}, n}(\bm{u}, \bm{v})\leq2r_{p, \bm{k}, n}}\cdot\mathbf{1}_{\|v\|_{p, \bm{k}, n}\leq\|u\|_{p, \bm{k}, n}}d\bm{v}d\bm{u}\\
&\leq2\max_{\bm{u}\in B^n_{p, \bm{k}}(2r_{p, \bm{k}, n})}\int_S\mathbf{1}_{d_{p, \bm{k}, n}(\bm{u}, \bm{v})\leq2r_{p, \bm{k}, n}}\cdot\mathbf{1}_{\|v\|_{p, \bm{k}, n}\leq\|u\|_{p, \bm{k}, n}}d\bm{v}\\
&\leq2\max_{\bm{u}\in B^n_{p, \bm{k}}(2r_{p, \bm{k}, n})}{\rm vol}\left(B^n_{p, \bm{k}}(\bm{u}, 2r_{p, \bm{k}, n})\cap B^n_{p, \bm{k}}(\bm{0}, \|\bm{u}\|_{p, \bm{k}, n})\right).
\end{split}
\end{equation*}
In order to prove inequality (\ref{xiaoyucp}), it suffices to prove that there exists a constant $c_p\in(0, 2)$ such that for every $n, \bm{k}$ and for every $\bm{u}\in B^n_{p, \bm{k}}(\bm{0}, 2r_{p, \bm{k}, n})$, we have
$$
{\rm vol}\left(B^n_{p, \bm{k}}(\bm{u}, 2r_{p, \bm{k}, n})\cap B^n_{p, \bm{k}}(\bm{0}, \|\bm{u}\|_{p, \bm{k}, n})\right)\leq c_p^n.
$$

For every $n, \bm{k}$ and for every $\bm{u}\in B^n_{p, \bm{k}}(\bm{0}, 2r_{p, \bm{k}, n})$, if $\|\bm{u}\|_{p, \bm{k}, n}\leq x_p\cdot r_{p, \bm{k}, n}$, then
$$
{\rm vol}\left(B^n_{p, \bm{k}}(\bm{u}, 2r_{p, \bm{k}, n})\cap B^n_{p, \bm{k}}(\bm{0}, \|\bm{u}\|_{p, \bm{k}, n})\right)\leq {\rm vol}\left(B^n_{p, \bm{k}}(\bm{0}, \|\bm{u}\|_{p, \bm{k}, n})\right)\leq x_p^n.
$$

Next we consider the case that $\|\bm{u}\|_{p, \bm{k}, n}\geq x_p\cdot r_{p, \bm{k}, n}$.
For every $\bm{x}\in B^n_{p, \bm{k}}(\bm{u}, 2r_{p, \bm{k}, n})\cap B^n_{p, \bm{k}}(\bm{0}, \|\bm{u}\|_{p, \bm{k}, n})$, we have $\|\bm{x}-\bm{u}\|_{p, \bm{k}, n}\leq2r_{p, \bm{k}, n}$ and $\|\bm{x}\|_{p, \bm{k}, n}\leq\|\bm{u}\|_{p, \bm{k}, n}\leq2r_{p, \bm{k}, n}$. If $\|\bm{x}-\bm{u}\|_{p, \bm{k}, n}\leq x_p\cdot r_{p, \bm{k}, n}$ or $\|\bm{x}\|_{p, \bm{k}, n}\leq x_p\cdot r_{p, \bm{k}, n}$, then by properties of a norm (the homogeneity condition and the triangle inequality), we have
\begin{equation}\label{qingkuang1}
\left\|\bm{x}-\frac{1}{2}\bm{u}\right\|_{p, \bm{k}, n}=\frac{1}{2}\left\|2\bm{x}-\bm{u}\right\|_{p, \bm{k}, n}\leq\frac{1}{2}\left(\left\|\bm{x}-\bm{u}\right\|_{p, \bm{k}, n}+\left\|\bm{x}\right\|_{p, \bm{k}, n}\right)\leq\frac{x_p+2}{2}r_{p, \bm{k}, n}.
\end{equation}

Now suppose that $\|\bm{x}-\bm{u}\|_{p, \bm{k}, n}\geq x_p\cdot r_{p, \bm{k}, n}$ and $\|\bm{x}\|_{p, \bm{k}, n}\geq x_p\cdot r_{p, \bm{k}, n}$. Fix $1<p\leq2$. Let $\delta_{p}(\epsilon)=1-\left(1-\left(\frac{\epsilon}{2}\right)^q\right)^{1/q}$. By Proposition \ref{yizhitu}, for every $\bm{k}=(k_1, k_2, \ldots)$ and for every $\bm{x}, \bm{y}\in\ell_{p, \bm{k}}$, if $\|\bm{x}\|_{p, \bm{k}}=\|\bm{y}\|_{p, \bm{k}}=1$ and $\|\bm{x}-\bm{y}\|_{p, \bm{k}}\geq\epsilon$, then $\left\|\frac{\bm{x}+\bm{y}}{2}\right\|_{p, \bm{k}}\leq1-\delta_p(\epsilon)$. For every $n\in\mathbb{N}$ and $\bm{k}=(k_1, k_2, \ldots, k_{m+1})$ with $0=k_1<k_2<\cdots<k_{m+1}=n$, we view a point $\bm{x}\in\mathbb{R}^n$ as a point in $\ell_{p, \tilde{\bm{k}}}$ by adding zeros after the $n$-th coordinate of $\bm{x}$ and $\tilde{\bm{k}}=(k_1, k_2, \ldots, k_{m+1}, n+1, n+2, \ldots)$, so $\|\bm{x}\|_{p, \bm{k}, n}=\|\bm{x}\|_{p, \tilde{\bm{k}}}$. In the rest of proof, we use norm $\|\cdot\|_{p, \bm{k}}$ instead of $\|\cdot\|_{p, \bm{k}, n}$.

Let $\|\bm{x}-\bm{u}\|_{p, \bm{k}}=a$, $\|\bm{x}\|_{p, \bm{k}}=b$, and $\|\bm{u}\|_{p, \bm{k}}=c$. We have known that $x_p\cdot r_{p, \bm{k}, n}\leq a\leq2r_{p, \bm{k}, n}$ and $x_p\cdot r_{p, \bm{k}, n}\leq b\leq c\leq2r_{p, \bm{k}, n}$. Since $\left\|\frac{\bm{x}-\bm{u}}{a}\right\|_{p, \bm{k}}=\left\|\frac{\bm{x}}{b}\right\|_{p, \bm{k}}=1$ and
$$
\left\|\frac{\bm{x}-\bm{u}}{a}-\frac{\bm{x}}{b}\right\|_{p, \bm{k}}=\left\|\frac{-\bm{u}}{a}+\left(\frac{1}{a}-\frac{1}{b}\right)\bm{x}\right\|_{p, \bm{k}}\geq\left\|\frac{-\bm{u}}{a}\right\|_{p, \bm{k}}-\left\|\left(\frac{1}{a}-\frac{1}{b}\right)\bm{x}\right\|_{p, \bm{k}}=\frac{c}{a}-\left|\frac{b-a}{a}\right|\geq\frac{x_p}{2}-\frac{|b-a|}{a},
$$
it follows from the uniform convexity of $\ell_{p, \bm{k}}$ that
\begin{equation}\label{yizhijielun}
\left\|\frac{\frac{\bm{x}-\bm{u}}{a}+\frac{\bm{x}}{b}}{2}\right\|_{p, \bm{k}}\leq1-\delta_p(\epsilon_p),
\end{equation}
where $\epsilon_p:=\frac{x_p}{2}-\frac{|b-a|}{a}\geq \frac{x_p}{2}-\frac{2-x_p}{x_p}=1+\frac{x_p}{2}-\frac{2}{x_p}$. So $\delta_p(\epsilon_p)\geq\delta_p\left(1+\frac{x_p}{2}-\frac{2}{x_p}\right)$.
Multiplying by $a$ in the both sides of the inequality (\ref{yizhijielun}) and rearranging it, we have
$$
\left\|\frac{1+\frac{a}{b}}{2}\bm{x}-\frac{1}{2}\bm{u}\right\|_{p, \bm{k}}\leq a(1-\delta_p(\epsilon_p)).
$$
Therefore,
\begin{equation}\label{qingkuang2}
\begin{split}
\left\|\bm{x}-\frac{1}{2}\bm{u}\right\|_{p, \bm{k}}&=\left\|\frac{1+\frac{a}{b}}{2}\bm{x}-\frac{1}{2}\bm{u}+\frac{1-\frac{a}{b}}{2}\bm{x}\right\|_{p, \bm{k}}\\
&\leq\left\|\frac{1+\frac{a}{b}}{2}\bm{x}-\frac{1}{2}\bm{u}\right\|_{p, \bm{k}}+\left\|\frac{1-\frac{a}{b}}{2}\bm{x}\right\|_{p, \bm{k}}\\
&\leq a(1-\delta_p(\epsilon_p))+\frac{|b-a|}{2}\\
&\leq \left(2(1-\delta_p(\epsilon_p))+\frac{2-x_p}{2}\right)r_{p, \bm{k}, n}\\
&=\left[2-\left(2\delta_p(\epsilon_p)-\frac{2-x_p}{2}\right)\right]r_{p, \bm{k}, n}.
\end{split}
\end{equation}
Since
\begin{equation*}
\begin{split}
2\delta_p\left(\epsilon_p\right)-\frac{2-x_p}{2}&=2\left[1-\left(1-\left(\frac{\epsilon_p}{2}\right)^q\right)^{1/q}\right]-1+\frac{x_p}{2}\\
&=1-2\left(1-\left(\frac{\epsilon_p}{2}\right)^q\right)^{1/q}+\frac{x_p}{2}\\
&\geq1-2\left(1-\left(\frac{1+\frac{x_p}{2}-\frac{2}{x_p}}{2}\right)^q\right)^{1/q}+\frac{x_p}{2}\\
&=1-2\left(1-\left(\frac{1}{2}+\frac{x_p}{4}-\frac{1}{x_p}\right)^q\right)^{1/q}+\frac{x_p}{2}
\end{split}
\end{equation*}
and by the definition of $x_p$
$$
1-\left(\frac{1}{2}+\frac{x_p}{4}-\frac{1}{x_p}\right)^q\leq\left(\frac{x_p+2}{4}\right)^q-\frac{1}{3^q},
$$
it follows that
\begin{equation}\label{dayuling}
\begin{split}
2\delta_p\left(\epsilon_p\right)-\frac{2-x_p}{2}&\geq1-2\left(1-\left(\frac{1}{2}+\frac{x_p}{4}-\frac{1}{x_p}\right)^q\right)^{1/q}+\frac{x_p}{2}\\
&\geq1+\frac{x_p}{2}-\left(\left(1+\frac{x_p}{2}\right)^q-\frac{2^q}{3^q}\right)^{1/q}\\
&>0.
\end{split}
\end{equation}

Let $c'_p=\max\left\{\frac{x_p+2}{2}, 2-\left(2\delta_p(\epsilon_p)-\frac{2-x_p}{2}\right)\right\}$. By the inequality (\ref{dayuling}) we see that $c'_p<2$. By inequalities (\ref{qingkuang1}) and (\ref{qingkuang2}) we see that
$$
B^n_{p, \bm{k}}(\bm{u}, 2r_{p, \bm{k}, n})\cap B^n_{p, \bm{k}}(\bm{0}, \|\bm{u}\|_{p, \bm{k}, n})\subseteq B^n_{p, \bm{k}}(\bm{u}/2, c'_pr_{p, \bm{k}, n}).
$$
So
$$
{\rm vol}\left(B^n_{p, \bm{k}}(\bm{u}, 2r_{p, \bm{k}, n})\cap B^n_{p, \bm{k}}(\bm{0}, \|\bm{u}\|_{p, \bm{k}, n})\right)\leq{\rm vol}\left(B^n_{p, \bm{k}}(\bm{u}/2, c'_pr_{p, \bm{k}, n})\right)=(c'_p)^n.
$$
Let $c_p=\max\{x_p, c'_p\}$ and we are done.

The equation (\ref{dengshi2}) follows from inequalities (\ref{dengshi3}) and (\ref{xiaoyucp}).
\end{proof}

Now we are able to prove Theorem \ref{main}.

\begin{proof}[Proof of Theorem \ref{main}]
Let $S\subseteq\mathbb{R}^n$ be bounded, measurable, and of positive volume. Let $c_p$ be the constant in Lemma \ref{zhongyao} and $\alpha=\alpha_{S, p, \bm{k}}(\lambda)$. Then by Jensen's inequality, we have
$$
\alpha=\lambda\cdot\mathbb{E}\left[\frac{1}{Z_{{\rm T}, p, \bm{k}}(\lambda)}\right]\geq\lambda\cdot e^{-\mathbb{E}\log Z_{{\rm T}, p, \bm{k}}(\lambda)},
$$
where the above expectation is with respect to the two-part experiment in forming the random set ${\rm T}$ and the first equality follows from the equation (\ref{qwer1}).

On the other hand, we have
\begin{equation*}
\begin{split}
\alpha&\geq2^{-n}\cdot\mathbb{E}\left[\frac{\lambda\cdot Z'_{{\rm T},p, \bm{k}}(\lambda)}{Z_{{\rm T}, p, \bm{k}}(\lambda)}\right]\\
&=2^{-n}\cdot\mathbb{E}\left[{\rm vol(T)}\cdot\alpha_{{\rm T}, p, \bm{k}}(\lambda)\right]\\
&\geq2^{-n}\cdot\mathbb{E}\left[\lambda\cdot{\rm vol(T)}\cdot e^{-\lambda\cdot2\cdot c_p^n}\right]\\
&\geq2^{-n}\cdot\mathbb{E}\left[\log Z_{{\rm T}, p, \bm{k}}(\lambda)\cdot e^{-\lambda\cdot2\cdot c_p^n}\right]\\
&=2^{-n}\cdot e^{-\lambda\cdot2\cdot c_p^n}\cdot\mathbb{E}\left[\log Z_{{\rm T}, p, \bm{k}}(\lambda)\right].
\end{split}
\end{equation*}
Combining these two lower bounds and letting $z=\mathbb{E}\left[\log Z_{{\rm T}, p, \bm{k}}(\lambda)\right]$, we have
$$
\alpha\geq\inf_z\max\left\{\lambda\cdot e^{-z}, 2^{-n}\cdot e^{-\lambda\cdot2\cdot c_p^n}\cdot z\right\}.
$$
Since $\lambda\cdot e^{-z}$ is decreasing in $z$ and $2^{-n}\cdot e^{-\lambda\cdot2\cdot c_p^n}\cdot z$ is increasing in $z$, the infimum over $z$ of the maximum of the two expressions occurs when they are equal, i.e., $\alpha\geq\lambda e^{-z^*}$, where $z^*$ is the solution to
$$
\lambda\cdot e^{-z}=2^{-n}\cdot e^{-\lambda\cdot2\cdot c_p^n}\cdot z.
$$
In other words,
\begin{equation}\label{zxing}
z^*=W(\lambda2^{n}e^{2\lambda c_p^n}),
\end{equation}
where $W(x)$ is the Lambert-W function. For $x>0$, $w=W(x)$ is defined to be the unique solution to the equation $we^w=x$. So
$$
w=\log x-\log(\log x-\log w)=\log x-\log\log x-\log\left(1-\frac{\log w}{\log x}\right).
$$
As $x\rightarrow\infty$,
$$
W(x)=\log x-\log\log x+O\left(\frac{\log\log x}{\log x}\right).
$$

We take $\lambda=n^{-1}c_p^{-n}$. So $\lambda2^{n}e^{2\lambda c_p^n}=\frac{1}{n}\left(\frac{2}{c_p}\right)^ne^{2/n}\rightarrow\infty$ as $n\rightarrow\infty$. The equation (\ref{zxing}) becomes
\begin{equation}
\begin{split}
z^*&=W(\lambda2^{n}e^{2/n})\\
&=\log \lambda+n\log2-\log n-\log\log(2/c_p)+O(\log (\log n/n)).
\end{split}
\end{equation}
So
$$
\alpha\geq\lambda e^{-z^*}=\left(1+O(\log n/n)\right)\frac{\log(2/c_p)\cdot n}{2^n}.
$$
Since $\alpha$ is increasing in $\lambda$, this bound holds for every $\lambda\geq n^{-1}c_p^{-n}$, completing the proof.
\end{proof}

In the above proof, if we take $\lambda=n^{-1}c^{-n}$ for some $c\in[c_p, 2)$, then we have
\begin{equation*}
\alpha_{S, p, \bm{k}}(e^{-n\log c})=\alpha_{S, p, \bm{k}}(c^{-n})\geq\alpha_{S, p, \bm{k}}(n^{-1}c^{-n})\geq\left(1+o_n(1)\right)\frac{(\log2-\log c)\cdot n}{2^n},
\end{equation*}
i.e.,
\begin{equation}\label{alphaxiajie}
\alpha_{S, p, \bm{k}}(e^{-nt})\geq\left(1+o_n(1)\right)\frac{(\log2-t)\cdot n}{2^n}
\end{equation}
for $t\in[\log c_p, \log2)$.

\section{An alternate proof of Theorem \ref{zhuyaojieguo}}\label{dierzhong}
In this section, we give an alternate proof of Theorem \ref{zhuyaojieguo} without using the hard superball model. The idea of the proof comes from \cite[Section 6.1]{2021arXiv211211274K}. We fix $p\in(1, 2]$ and $\bm{k}$, and let $n$ be a sufficiently large number. For convenience, in this section, we simplify most of our symbols and variables: $B^n(R)=B^n(\bm{0}, R):=B^n_{p, \bm{k}}(R)=B^n_{p, \bm{k}}(\bm{0}, R), r_{n}:=r_{p, \bm{k}, n}, B^n(\bm{x}, r_{n}):=B^n_{p, \bm{k}}(\bm{x}, r_{p, \bm{k}, n}), \|\cdot\|:=\|\cdot\|_{p, \bm{k}, n}$, and $d(\cdot, \cdot):=d_{p, \bm{k}, n}(\cdot, \cdot)$.

Let $B^n(R)$ be the superball centered at $\bm{0}$ with radius $R$. Consider the packings in $B^n(R)$ using $B^n(\bm{x}, r_{n})$. Let $\epsilon$ be a small positive real number (we will determine $\epsilon$ later), ${\rm C}_\epsilon:=[0, \epsilon]^n$ be a basic cube and $\tilde{L}_\epsilon:=(\epsilon\mathbb{Z})^n$ be a lattice. ${\rm C}_\epsilon+\tilde{L}_\epsilon$ tiles the whole space, and it also gives a partition of $B^n(R)$. Let $L_\epsilon=\{\bm{x}\in \tilde{L}_\epsilon: {\rm C}_\epsilon+\bm{x}\subseteq B^n(R)\}$. We have the following lemma.
\begin{lemma}\label{cover}
$\{{\rm C}_\epsilon+\bm{x}:\bm{x}\in L_\epsilon\}$ covers $B^n(R-2n^{\frac{p+2}{2p}}\epsilon)$.
\end{lemma}
\begin{proof}
Suppose on the contrary that there is some $\tilde{\bm{y}}=(y_1, y_2, \ldots, y_n)\in B^n(R-2n^{\frac{p+2}{2p}}\epsilon)$ such that $\bm{y}\notin{\rm C}_\epsilon+\bm{x}$ for every $\bm{x}\in L_\epsilon$. Let $\tilde{\bm{x}}=(x_1, x_2, \ldots, x_n)\in \tilde{L}_\epsilon$ such that $0\leq y_i-x_i\leq\epsilon$ for every $1\leq i\leq n$, in other words, $\tilde{\bm{y}}\in{\rm C}_\epsilon+\tilde{\bm{x}}$. By the choice of $\tilde{\bm{y}}$, we have that ${\rm C}_\epsilon+\tilde{\bm{x}}\nsubseteq B^n(R)$. On the other hand, for every $\bm{z}\in{\rm C}_\epsilon+\tilde{\bm{x}}$, we have
$$
\|\bm{z}\|\leq\|\bm{z}-\bm{y}\|+\|\bm{y}\|\leq\left(\sum_{j=1}^{m}\left\|\left(\epsilon, \epsilon, \ldots, \epsilon\right)\right\|_{2}^p\right)^{1/p}+R-2n^{\frac{p+2}{2p}}\epsilon.
$$
Note that $\left\|\left(\epsilon, \epsilon, \ldots, \epsilon\right)\right\|_{2}^p=(\epsilon^2+\epsilon^2+\cdots+\epsilon^2)^{p/2}\leq n^{p/2}\epsilon^p$ since the number of $\epsilon$ is at most $n$, and $\sum_{j=1}^{m}\left\|\left(\epsilon, \epsilon, \ldots, \epsilon\right)\right\|_{2}^p\leq n^{(p+2)/2}\epsilon^p$ since $m\leq n$. So
$$
\|\bm{z}\|\leq\left(\sum_{j=1}^{m}\left\|\left(\epsilon, \epsilon, \ldots, \epsilon\right)\right\|_{2}^p\right)^{1/p}+R-2n^{\frac{p+2}{2p}}\epsilon\leq R-2n^{\frac{p+2}{2p}}\epsilon+n^{\frac{p+2}{2p}}\epsilon<R.
$$
This holds for every $\bm{z}\in{\rm C}_\epsilon+\tilde{\bm{x}}$. Thus ${\rm C}_\epsilon+\tilde{\bm{x}}\subseteq B^n(R)$, a contradiction.
\end{proof}

If we write $N:=\left|L_\epsilon\right|$, we can estimate $N$ by Lemma \ref{cover}; that is,
$$
{\rm vol}\left(B^n(R-2n^{\frac{p+2}{2p}}\epsilon)\right)\leq N\cdot{\rm vol}({\rm C}_\epsilon)\leq{\rm vol}\left(B^n(R)\right),
$$
i.e.,
$$
\left(\frac{R-2n^{\frac{p+2}{2p}}\epsilon}{\epsilon r_{n}}\right)^n\leq N\leq\left(\frac{R}{\epsilon r_{n}}\right)^n.
$$
So $N=(1-o_R(1))\left(\frac{R}{\epsilon r_{n}}\right)^n$.

Let $\{{\rm C}_\epsilon+\bm{x}:\bm{x}\in L_\epsilon\}=\{C_1, C_2, \ldots, C_N\}$. We arbitrarily choose $\bm{v}_i$ from each $C_i$, and construct an auxiliary graph $G$ with vertex set $V(G)=\{\bm{v}_1, \bm{v}_2, \ldots, \bm{v}_N\}$, where $\bm{v}_i$ and $\bm{v}_j$ are adjacent if and only if $d(\bm{v}_i, \bm{v}_j)<2r_{n}$. So the maximum number of superballs with radius $r_{n}$ that can be packed in $B^n(R)$ is larger than or equal to $\alpha(G)$, where $\alpha(G)$ is the independence number of $G$. In other words,
\begin{equation}
\Delta_{p, \bm{k}}(n)\geq\lim_{R\rightarrow\infty}\frac{\alpha(G)}{\left(R/r_{n}\right)^n}.
\end{equation}
Here, using the same trick in the proof of Lemma \ref{xiajie}, the points near the boundary of $B^n(R)$ do not affect the result.

\begin{lemma}\label{baohan}
Every vertex in $G$ has degree at most $\frac{1+o_n(1)}{1-o_R(1)}\left(\frac{2r_{n}}{R}\right)^nN$.
\end{lemma}
\begin{proof}
We write $N[\bm{x}]=N(\bm{x})\cup\{\bm{x}\}$ for the closed neighborhood of $\bm{x}$. We claim that for every $\bm{x}\in V(G)$,
\begin{equation}\label{tiji}
B^n(\bm{x}, 2r_{n}-2n^{\frac{p+2}{2p}}\epsilon)\cap B^n(\bm{0}, R-2n^{\frac{p+2}{2p}}\epsilon)\subseteq\bigcup_{\bm{v}_i\in N[\bm{x}]}C_i\subseteq B^n(\bm{x}, 2r_{n}+2n^{\frac{p+2}{2p}}\epsilon).
\end{equation}
For the first inclusion, let $\bm{y}\in B^n(\bm{x}, 2r_{n}-2n^{\frac{p+2}{2p}}\epsilon)\cap B^n(\bm{0}, R-2n^{\frac{p+2}{2p}}\epsilon)$. By Lemma \ref{cover}, there exists an index $i$ such that $\bm{y}\in C_i$. So $d(\bm{y}, \bm{v}_i)\leq n^{\frac{p+2}{2p}}\epsilon$. As $d(\bm{x}, \bm{y})\leq2r_{n}-2n^{\frac{p+2}{2p}}\epsilon$, it follows that $d(\bm{x}, \bm{v}_i)\leq2r_{n}-2n^{\frac{p+2}{2p}}\epsilon+n^{\frac{p+2}{2p}}\epsilon<2r_{n}$, i.e. $\bm{v}_i\in N[\bm{x}]$. For the second inclusion, let $\bm{y}\in C_i$ for some index $i$ with $\bm{v}_i\in N[\bm{x}]$. We have $d(\bm{y}, \bm{v}_i)\leq n^{\frac{p+2}{2p}}\epsilon$ and $d(\bm{x}, \bm{v}_i)<2r_{p, \bm{k}, n}$. So $d(\bm{y}, \bm{x})\leq 2r_{p, \bm{k}, n}+2n^{\frac{p+2}{2p}}\epsilon$, as desired.

Since $C_i$'s are nonoverlapping, it follows from the claim that
$$
{\rm vol}\left(\bigcup_{\bm{v}_i\in N[\bm{x}]}C_i\right)=|N[\bm{x}]|\epsilon^n\leq{\rm vol}(B^n(\bm{x}, 2r_{n}+2n^{\frac{p+2}{2p}}\epsilon))=\left(\frac{2r_{n}+2n^{\frac{p+2}{2p}}\epsilon}{r_{n}}\right)^n,
$$
i.e.,
$$
|N[\bm{x}]|\leq\left(\frac{1}{\epsilon r_{n}}\right)^n(2r_{n}+2n^{\frac{p+2}{2p}}\epsilon)^n=\frac{1+o_n(1)}{1-o_R(1)}\left(\frac{2r_{n}}{R}\right)^nN,
$$
where we choose $\epsilon$ satisfying $n^{\frac{p+2}{2p}}\epsilon/r_n<n^{-2}$.
\end{proof}

Let $D:=\frac{1+o_n(1)}{1-o_R(1)}\left(\frac{2r_{n}}{R}\right)^nN$ and $K:=\frac{1}{10}\left(\frac{2}{c_p}\right)^n$, where $c_p$ is the constant in Lemma \ref{zhongyao}. Let $G[N(\bm{x})]$ denote the induced subgraph of $G$ whose vertex set is $N(\bm{x})$. We have the following lemma.
\begin{lemma}\label{pingjundu}
Suppose $R$ and $n$ are sufficiently large. For every $\bm{x}\in V(G)$, the average degree of $G[N(\bm{x})]$ is at most $\frac{D}{K}$.
\end{lemma}
\begin{proof}
We calculate $\frac{D}{K}=10(1+o_n(1))\left(\frac{c_p}{\epsilon}\right)^n$. Let $S'=\bigcup_{\bm{v}_i\in N(\bm{x})}C_i$ and $N(\bm{x})=\{\bm{x}_1, \bm{x}_2, \ldots, \bm{x}_t\}$ where $t=|N(\bm{x})|$. So ${\rm vol}(S')=t\epsilon^n$. By the definition, the average degree of $G[N(\bm{x})]$ is at most
\begin{equation}\label{37}
\frac{1}{t}\sum_{i, j=1}^t\mathbf{1}_{d(\bm{x}_i, \bm{x}_j)\leq2r_{n}}=\frac{\epsilon^n}{{\rm vol}(S')}\sum_{i, j=1}^t\mathbf{1}_{d(\bm{x}_i, \bm{x}_j)\leq2r_{n}}.
\end{equation}
By the definition of integration,
$$
\lim_{\epsilon\rightarrow0}\epsilon^{2n}\sum_{i, j=1}^t\mathbf{1}_{d(\bm{x}_i, \bm{x}_j)\leq2r_{n}}=\int_{S'}\int_{S'}\mathbf{1}_{d(\bm{u}, \bm{v})\leq2r_{n}}d\bm{u}d\bm{v},
$$
in other words, for every $\delta>0$, there exists $\epsilon_0(\delta)$ such that
\begin{equation}\label{38}
\epsilon^{2n}\sum_{i, j=1}^t\mathbf{1}_{d(\bm{x}_i, \bm{x}_j)\leq2r_{n}}\leq\int_{S'}\int_{S'}\mathbf{1}_{d(\bm{u}, \bm{v})\leq2r_{n}}d\bm{u}d\bm{v}+\delta.
\end{equation}
whenever $\epsilon<\epsilon_0(\delta)$.

Recall that in the proof of Lemma \ref{baohan} we have $S'\subseteq B^n(\bm{x}, 2r_{n}+2n^{\frac{p+2}{2p}}\epsilon)$. Let $S=\frac{2r_{n}}{2r_{n}+2n^{\frac{p+2}{2p}}\epsilon}S'$. So $S\subseteq B^n(\bm{x}, 2r_{n})$ and
\begin{equation}\label{39}
\begin{split}
\int_{S'}\int_{S'}\mathbf{1}_{d(\bm{u}, \bm{v})\leq2r_{n}}d\bm{u}d\bm{v}=&\left(\frac{2r_{n}+2n^{\frac{p+2}{2p}}\epsilon}{2r_{n}}\right)^{2n}\int_{S}\int_{S}\mathbf{1}_{d(\bm{u}, \bm{v})\leq2r_{n}\cdot2r_n/(2r_{n}+2n^{\frac{p+2}{2p}}\epsilon)}d\bm{u}d\bm{v}\\
\leq&\left(\frac{2r_{n}+2n^{\frac{p+2}{2p}}\epsilon}{2r_{n}}\right)^{n}\frac{{\rm vol}(S')}{{\rm vol}(S)}\int_{S}\int_{S}\mathbf{1}_{d(\bm{u}, \bm{v})\leq2r_{n}}d\bm{u}d\bm{v}\\
=&(1+o_n(1))\frac{{\rm vol}(S')}{{\rm vol}(S)}\int_{S}\int_{S}\mathbf{1}_{d(\bm{u}, \bm{v})\leq2r_{n}}d\bm{u}d\bm{v},
\end{split}
\end{equation}
where we choose $\epsilon$ small enough so that $\left(\frac{2r_{n}+2n^{\frac{p+2}{2p}}\epsilon}{2r_{n}}\right)^{n}=1+o_n(1)$. For instance, if we choose $\epsilon$ satisfying $n^{\frac{p+2}{2p}}\epsilon/r_{n}<n^{-2}$, then $\left(\frac{2r_{n}+2n^{\frac{p+2}{2p}}\epsilon}{2r_{n}}\right)^{n}<\left(1+\frac{1}{n^2}\right)^{n}<e^{1/n}=1+o_n(1)$.

Recall that in the proof of Lemma \ref{zhongyao}, we have
\begin{equation}\label{40}
\int_{S}\int_{S}\mathbf{1}_{d(\bm{u}, \bm{v})\leq2r_{n}}d\bm{u}d\bm{v}\leq2{\rm vol}(S)c_p^n.
\end{equation}
Combining equation (\ref{37}) and inequalities (\ref{38})-(\ref{40}), the average degree of $G[N(\bm{x})]$ is at most
\begin{equation}\label{42}
\begin{split}
\frac{\epsilon^n}{{\rm vol}(S')}\sum_{i, j=1}^t\mathbf{1}_{d(\bm{x}_i, \bm{x}_j)\leq2r_{n}}\leq&\frac{1}{{\rm vol}(S')\epsilon^n}\left(\int_{S'}\int_{S'}\mathbf{1}_{d(\bm{u}, \bm{v})\leq2r_{n}}d\bm{u}d\bm{v}+\delta\right)\\
\leq&\frac{1}{{\rm vol}(S')\epsilon^n}\left((1+o_n(1))\frac{{\rm vol}(S')}{{\rm vol}(S)}\int_{S}\int_{S}\mathbf{1}_{d(\bm{u}, \bm{v})\leq2r_{n}}d\bm{u}d\bm{v}+\delta\right)\\
\leq&\frac{1}{{\rm vol}(S')\epsilon^n}\left((1+o_n(1))\cdot2c_p^n{\rm vol}(S')+\delta\right)\\
=&(1+o_n(1))\cdot2\left(\frac{c_p}{\epsilon}\right)^n+\frac{\delta}{{\rm vol}(S')\epsilon^n}.
\end{split}
\end{equation}

Next we give a lower bound of ${\rm vol}(S')$. Recalling the inclusion relation (\ref{tiji}), we have
$$
\bigcup_{\bm{v}_i\in N[\bm{x}]}C_i \supseteq B^n(\bm{x}, 2r_{n}-2n^{\frac{p+2}{2p}}\epsilon)\cap B^n(\bm{0}, R-2n^{\frac{p+2}{2p}}\epsilon).
$$
Denote $V_{{\rm lower}}:={\rm vol}\left(B^n(\bm{x}, 2r_{n}-2n^{\frac{p+2}{2p}}\epsilon)\cap B^n(\bm{0}, R-2n^{\frac{p+2}{2p}}\epsilon)\right)$. If $B^n(\bm{x}, 2r_{n}-2n^{\frac{p+2}{2p}}\epsilon)\subseteq B^n(\bm{0}, R-2n^{\frac{p+2}{2p}}\epsilon)$, then
$$
V_{{\rm lower}}\geq{\rm vol}\left(B^n(\bm{x}, 2r_{n}-2n^{\frac{p+2}{2p}}\epsilon)\right)=\left(\frac{2r_{n}-2n^{\frac{p+2}{2p}}\epsilon}{r_{n}}\right)^n=(1-o_n(1))2^n.
$$
If $B^n(\bm{x}, 2r_{n}-2n^{\frac{p+2}{2p}}\epsilon)\nsubseteq B^n(\bm{0}, R-2n^{\frac{p+2}{2p}}\epsilon)$, then we have
$$
B^n(\bm{x}, 2r_{n}-2n^{\frac{p+2}{2p}}\epsilon)\cap B^n(\bm{0}, R-2n^{\frac{p+2}{2p}}\epsilon)\supseteq B^n\left(\bm{y}, r_{n}-2n^{\frac{p+2}{2p}}\epsilon\right),
$$
where $\bm{y}:=\left(1-\frac{r_{n}}{\|\bm{x}\|}\right)\bm{x}$. To see this, letting $\bm{z}\in B^n\left(\bm{y}, r_{n}-2n^{\frac{p+2}{2p}}\epsilon\right)$, we have $d(\bm{z}, \bm{y})\leq r_{n}-2n^{\frac{p+2}{2p}}\epsilon$. So
$$
d(\bm{x}, \bm{z})\leq d(\bm{x}, \bm{y})+d(\bm{y}, \bm{z})\leq r_n+r_{n}-2n^{\frac{p+2}{2p}}\epsilon=2r_{n}-2n^{\frac{p+2}{2p}}\epsilon.
$$
And
$$
d(\bm{0}, \bm{z})\leq d(\bm{0}, \bm{y})+d(\bm{y}, \bm{z})\leq \|\bm{x}\|-r_n+r_{n}-2n^{\frac{p+2}{2p}}\epsilon\leq R-2n^{\frac{p+2}{2p}}\epsilon,
$$
Since $\bm{x}\in B^n(\bm{0}, R)$. Thus
$$
V_{{\rm lower}}\geq {\rm vol}\left(B^n\left(\bm{y}, r_{n}-2n^{\frac{p+2}{2p}}\epsilon\right)\right)=\left(\frac{r_{n}-2n^{\frac{p+2}{2p}}\epsilon}{r_{n}}\right)^n=1-o_n(1).
$$
In conclusion, we have
$$
{\rm vol}(S')={\rm vol}\left(\bigcup_{\bm{v}_i\in N[\bm{x}]}C_i\right)-\epsilon^n\geq V_{{\rm lower}}-\epsilon^n\geq1-o_n(1).
$$

Therefore, if we choose $\delta=c_p^n$ and $\epsilon\leq\min\{\epsilon_0(c_p^n), r_n/n^{2+\frac{p+2}{2p}}\}$, then by inequality (\ref{42}), we know that the average degree of $G[N(\bm{x})]$ is at most
$$
(1+o_n(1))\cdot2\left(\frac{c_p}{\epsilon}\right)^n+\frac{\delta}{{\rm vol}(S')\epsilon^n}\leq(1+o_n(1))\cdot2\left(\frac{c_p}{\epsilon}\right)^n+\frac{c_p^n}{(1-o_n(1))\epsilon^n}<5\left(\frac{c_p}{\epsilon}\right)^n<\frac{D}{K},
$$
for large $n$. Note that this is independent of the choice of $\bm{x}$, so we are done.
\end{proof}

By Lemmas \ref{baohan} and \ref{pingjundu}, we know that the $N$-vertex graph $G$ has maximum degree $D$ and every subgraph induced by a neighborhood in $G$ has average degree at most $D/K$. We say such a graph is locally sparse. By a result of Hurley and Pirot \cite[Corollary 1]{2021arXiv210915215P}, the chromatic number of $G$ is at most $(1+o_K(1))\frac{D}{\log K}$. So the independence number $\alpha(G)$ of $G$ is at least
\begin{equation}
\begin{split}
(1-o_K(1))\frac{N}{D}\log K=&(1-o_n(1))\frac{1-o_R(1)}{1+o_n(1)}\left(\frac{R}{2r_{n}}\right)^n\log\left(\frac{1}{10}\left(\frac{2}{c_p}\right)^n\right)\\
=&(1-o_n(1))(1-o_R(1))\left(\frac{R}{r_{n}}\right)^n\frac{\log(2/c_p)\cdot n}{2^n},
\end{split}
\end{equation}
and
$$
\Delta_{p, \bm{k}}(n)\geq\lim_{R\rightarrow\infty}\frac{\alpha(G)}{\left(R/r_{n}\right)^n}=(1-o_n(1))\frac{\log(2/c_p)\cdot n}{2^n}.
$$
This completes the proof of Theorem \ref{zhuyaojieguo}.

\section{A lower bound on entropy density and pressure}\label{entropy}
In this section, we investigate the entropy density and pressure of packings. Let $B^n_{p, \bm{k}}(R)$ be the region of packings and $V={\rm vol}(B^n_{p, \bm{k}}(R))=(R/r_{p, \bm{k}, n})^n$. Define
$$
f_{p, \bm{k}, n}(\alpha)=\lim_{V\rightarrow\infty}\frac{1}{\alpha V}\log\frac{\hat{Z}_{B^n_{p, \bm{k}}(R), p, \bm{k}}(\lfloor\alpha V\rfloor)}{V^{\lfloor\alpha V\rfloor}/\lfloor\alpha V\rfloor!}
$$
to be the \textit{entropy density} of the packings. Consider a random packing in $B^n_{p, \bm{k}}(R)$ with density $\alpha$. The superballs we use in the packings have volume $1$, so there are $\lfloor\alpha V\rfloor$ centers. $V^{\lfloor\alpha V\rfloor}/\lfloor\alpha V\rfloor!$ measures
all the choices of $\lfloor\alpha V\rfloor$ unordered points chosen in $B^n_{p, \bm{k}}(R)$, and $\hat{Z}_{B^n_{p, \bm{k}}(R), p, \bm{k}}(\lfloor\alpha V\rfloor)$ measures the choices of $\lfloor\alpha V\rfloor$ unordered points that can form a packing. Dividing by $\alpha V$ ensures that $f_{p, \bm{k}, n}(\alpha)$ is independent of the choice of the size of superballs in our packings. So $f_{p, \bm{k}, n}(\alpha)$ measures how plentiful the packings of density $\alpha$ are. We also define
$$
g_{p, \bm{k}, n}(\lambda)=\lim_{V\rightarrow\infty}\frac{1}{V}\log Z_{B^n_{p, \bm{k}}(R), p, \bm{k}}(\lambda)
$$
to be the \textit{pressure} of the packings. $g_{p, \bm{k}, n}(\lambda)$ measures how plentiful the packings with fugacity $\lambda$ are. We have the following theorems.

\begin{theorem}\label{gxiajie}
Let $c_p$ be the constant in Theorem \ref{main}. For $\lambda\in(2^{-n}, c_p^{-n}]$, we have
\begin{equation}\label{xiajieg}
g_{p, \bm{k}, n}(\lambda)\geq\left(\frac{\left(\log2+\frac{1}{n}\log\lambda\right)^2}{2}+o_n(1)\right)\frac{n^2}{2^n}.
\end{equation}
\end{theorem}
Note that in Theorem \ref{gxiajie}, $\log2+\frac{1}{n}\log\lambda\in(0, \log2-\log c_p]$.

\begin{theorem}\label{fxiajie}
Let $c_p$ be the constant in Theorem \ref{main}. There exists $\alpha=(1+o_{n}(1))\frac{\log(2/c_p)\cdot n}{2^n}$ such that
$$
f_{p, \bm{k}, n}(\alpha)\geq-(1+o_n(1))\log(2/c_p)\cdot n
$$
\end{theorem}

The proofs of Theorems \ref{gxiajie} and \ref{fxiajie} are similar to those of \cite[Theorem 4]{MR3898718} and \cite[Theorem 5]{MR3898718}, respectively. For the sake of completeness, we include them here.

\begin{proof}[Proof of Theorem \ref{gxiajie}]
We calculate
\begin{equation*}
\begin{split}
\frac{1}{V}\log Z_{B^n_{p, \bm{k}}(R), p, \bm{k}}(\lambda)=&\lim_{\epsilon\rightarrow0}\int_\epsilon^\lambda\frac{1}{V}\left(\log Z_{B^n_{p, \bm{k}}(R), p, \bm{k}}(x)\right)'dx\\
=&\lim_{\epsilon\rightarrow0}\int_\epsilon^\lambda\frac{\alpha_{B^n_{p, \bm{k}}(R), p, \bm{k}}(x)}{x}dx\\
\geq&\int_{2^{-n}}^\lambda\frac{\alpha_{B^n_{p, \bm{k}}(R), p, \bm{k}}(x)}{x}dx,
\end{split}
\end{equation*}
where the second equation uses equation (\ref{dengshi4}) and the last inequality follows from the positivity of the integrand. Put $x=e^{-nt}$. When $x=\lambda$, $t=-\frac{1}{n}\log\lambda$; when $x=2^{-n}$, $t=\log2$. Since $\lambda\in(2^{-n}, c_p^{-n}]$, it follows that $[-\frac{1}{n}\log\lambda, \log2]\subseteq[\log c_p, \log2]$ and we can use inequality (\ref{alphaxiajie}) to estimate the lower bound of $\alpha_{B^n_{p, \bm{k}}(R), p, \bm{k}}(e^{-nt})$. We have
\begin{equation*}
\begin{split}
\int_{2^{-n}}^\lambda\frac{\alpha_{B^n_{p, \bm{k}}(R), p, \bm{k}}(x)}{x}dx=&-n\int_{\log2}^{-\frac{1}{n}\log\lambda}\alpha_{B^n_{p, \bm{k}}(\bm{0}, R), p, \bm{k}}(e^{-nt})dt\\
=&n\int_{-\frac{1}{n}\log\lambda}^{\log2}\alpha_{B^n_{p, \bm{k}}(R), p, \bm{k}}(e^{-nt})dt\\
\geq&n\left(1+o(1)\right)\int_{-\frac{1}{n}\log\lambda}^{\log2}\frac{(\log2-t)\cdot n}{2^n}dt\\
=&\left(\frac{\left(\log2+\frac{1}{n}\log\lambda\right)^2}{2}+o_n(1)\right)\frac{n^2}{2^n}.
\end{split}
\end{equation*}
Let $V$ go to infinity and we obtain equation (\ref{xiajieg}).
\end{proof}

Before proving Theorem \ref{fxiajie}, we give a simple fact of $f_{p, \bm{k}, n}(\alpha)$: the larger the density is, the less such packings exist.

\begin{lemma}\label{dijian}
$f_{p, \bm{k}, n}(\alpha)$ is decreasing in $\alpha$.
\end{lemma}
\begin{proof}
If $\alpha>\Delta_{p, \bm{k}}(n)$, then $\hat{Z}_{B^n_{p, \bm{k}}(R), p, \bm{k}}(\lfloor\alpha V\rfloor)=0$ for large $V$. So $f_{p, \bm{k}, n}(\alpha)=0$ when $\alpha>\Delta_{p, \bm{k}}(n)$.

Suppose $0<\alpha_1<\alpha_2\leq\Delta_{p, \bm{k}}(n)$, and let $V_1$ and $V_2$ satisfy $\alpha_1V_1=\alpha_2V_2$ (so $V_1>V_2$). We claim that
\begin{equation}\label{alphav}
\frac{1}{\alpha_1 V_1}\log\frac{\hat{Z}_{B^n_{p, \bm{k}}(R_1), p, \bm{k}}(\lfloor\alpha_1 V_1\rfloor)}{V_1^{\lfloor\alpha_1 V_1\rfloor}/\lfloor\alpha_1 V_1\rfloor!}>\frac{1}{\alpha_2 V_2}\log\frac{\hat{Z}_{B^n_{p, \bm{k}}(R_2), p, \bm{k}}(\lfloor\alpha_2 V_2\rfloor)}{V_2^{\lfloor\alpha_2 V_2\rfloor}/\lfloor\alpha_2 V_2\rfloor!},
\end{equation}
where $V_1=(R_1/r_{p, \bm{k}, n})^n$ and $V_2=(R_2/r_{p, \bm{k}, n})^n$ (so $R_1>R_2$).
In order to prove inequality (\ref{alphav}), it suffices to prove
$$
\frac{\hat{Z}_{B^n_{p, \bm{k}}(R_1), p, \bm{k}}(\lfloor\alpha_1 V_1\rfloor)}{V_1^{\lfloor\alpha_1 V_1\rfloor}}>\frac{\hat{Z}_{B^n_{p, \bm{k}}(R_2), p, \bm{k}}(\lfloor\alpha_2 V_2\rfloor)}{V_2^{\lfloor\alpha_2 V_2\rfloor}}.
$$
Using equation (\ref{defzhat}), the definition of $\hat{Z}_{S, p, \bm{k}}(t)$, it reduces to
\begin{equation}\label{gailv}
\frac{\int_{B^n_{p, \bm{k}}(R_1)^t}\textbf{1}_{\mathcal{D}_{p, \bm{k}}(\bm{x}_1, \bm{x}_2, \ldots, \bm{x}_t)}d\bm{x}_1d\bm{x}_2\cdots d\bm{x}_t}{V_1^{t}}>\frac{\int_{B^n_{p, \bm{k}}(R_2)^t}\textbf{1}_{\mathcal{D}_{p, \bm{k}}(\bm{x}_1, \bm{x}_2, \ldots, \bm{x}_t)}d\bm{x}_1d\bm{x}_2\cdots d\bm{x}_t}{V_2^{t}},
\end{equation}
where $t=\lfloor\alpha_1V_1\rfloor=\lfloor\alpha_2V_2\rfloor$. Consider
$$
\frac{\int_{B^n_{p, \bm{k}}(R_2)^t}\textbf{1}_{\mathcal{D}_{p, \bm{k}}(\bm{x}_1, \bm{x}_2, \ldots, \bm{x}_t)}d\bm{x}_1d\bm{x}_2\cdots d\bm{x}_t}{V_2^{t}},
$$
the right hand side of inequality (\ref{gailv}). Let $\bm{y}_i=(R_1/R_2)\bm{x}_i, i=1, \ldots, t$. Then
$$
\frac{\int_{B^n_{p, \bm{k}}(R_2)^t}\textbf{1}_{\mathcal{D}_{p, \bm{k}}(\bm{x}_1, \bm{x}_2, \ldots, \bm{x}_t)}d\bm{x}_1d\bm{x}_2\cdots d\bm{x}_t}{V_2^{t}}=\frac{\int_{B^n_{p, \bm{k}}(R_1)^t}\textbf{1}_{\mathcal{D}'_{p, \bm{k}}(\bm{y}_1, \bm{y}_2, \ldots, \bm{y}_t)}d\bm{y}_1d\bm{y}_2\cdots d\bm{y}_t}{V_1^{t}},
$$
where $\mathcal{D}'_{p, \bm{k}}(\bm{y}_1, \bm{y}_2, \ldots, \bm{y}_t)$ means the event that $d_{p, \bm{k}, n}(\bm{y}_i, \bm{y}_j)\geq\frac{R_1}{R_2}\cdot2r_{p, \bm{k}, n}$ for every $i\neq j$.

The left hand side of inequality (\ref{gailv}) is the probability of the event that $t$ uniformly chosen points in $B^n_{p, \bm{k}}(R_1)$ with pairwise distance larger than or equal to $2r_{p, \bm{k}, n}$, while the right hand side of inequality (\ref{gailv}) is the probability of the event that $t$ uniformly chosen points in $B^n_{p, \bm{k}}(R_1)$ with pairwise distance larger than or equal to $\frac{R_1}{R_2}\cdot2r_{p, \bm{k}, n}$. The latter event is contained in the former event, i.e.
$$
\textbf{1}_{\mathcal{D}_{p, \bm{k}}(\bm{x}_1, \bm{x}_2, \ldots, \bm{x}_t)}\geq\textbf{1}_{\mathcal{D}'_{p, \bm{k}}(\bm{x}_1, \bm{x}_2, \ldots, \bm{x}_t)}.
$$
Thus inequality (\ref{gailv}) holds and the claim is proved.

Letting $V_1$ go to infinity in both sides of inequality (\ref{alphav}), we obtain that
$$
f_{p, \bm{k}, n}(\alpha_1)>f_{p, \bm{k}, n}(\alpha_2).
$$
This completes the proof.
\end{proof}

Finally, we give a proof of Theorem \ref{fxiajie}.
\begin{proof}[Proof of Theorem \ref{fxiajie}]
Consider packings in $B^n_{p, \bm{k}}(R)$ and $V={\rm vol}(B^n_{p, \bm{k}}(R))=(R/r_{p, \bm{k}, n})^n$. Assume that ${\rm Var}(|X|)\geq V^{3/2}$ for every $\lambda\in[c_p^{-n}, 2c_p^{-n}]$. Using equation (\ref{fangcha}) in the proof of Lemma \ref{dizeng}, we have
\begin{equation*}
\begin{split}
\alpha_{B^n_{p, \bm{k}}(R), p, \bm{k}}(2c_p^{-n})=&\int_0^{2c_p^{-n}}\alpha'_{B^n_{p, \bm{k}}(R), p, \bm{k}}(x)dx\\
=&\int_0^{2c_p^{-n}}\frac{{\rm Var}(|X|)}{xV}dx\\
\geq&\int_{c_p^{-n}}^{2c_p^{-n}}\frac{{\rm Var}(|X|)}{xV}dx\\
\geq&V^{1/2}(\log 2-1)>1,
\end{split}
\end{equation*}
a contradiction.
So there exists $\lambda\in[c_p^{-n}, 2c_p^{-n}]$ such that ${\rm Var}(|X|)<V^{3/2}$.

Choose $\lambda\in[c_p^{-n}, 2c_p^{-n}]$ such that ${\rm Var}(|X|)<V^{3/2}$. Using Chebyshev's inequality, we have
$$
\mathbb{P}_{B^n_{p, \bm{k}}(R), p, \bm{k}, \lambda}\left(\left||X|-\mathbb{E}_{B^n_{p, \bm{k}}(R), p, \bm{k}, \lambda}(|X|)\right|\geq V^{4/5}\right)\leq\frac{{\rm Var}(|X|)}{V^{8/5}}<V^{-1/10}.
$$
So
$$
\mathbb{P}_{B^n_{p, \bm{k}}(R), p, \bm{k}, \lambda}\left(\left||X|-\mathbb{E}_{B^n_{p, \bm{k}}(R), p, \bm{k}, \lambda}(|X|)\right|\leq V^{4/5}\right)\geq1-V^{-1/10}.
$$
Note that $|X|$ is an integer. By averaging, there exists an integer
$$
t\in\left[\mathbb{E}_{B^n_{p, \bm{k}}(R), p, \bm{k}, \lambda}(|X|)-V^{4/5}, \mathbb{E}_{B^n_{p, \bm{k}}(R), p, \bm{k}, \lambda}(|X|)+V^{4/5}\right]
$$
such that
$$
\mathbb{P}_{B^n_{p, \bm{k}}(R), p, \bm{k}, \lambda}\left(|X|=t\right)\geq\frac{1-V^{-1/10}}{\lfloor2V^{4/5}\rfloor}\geq \frac1V,
$$
for large $V$. Recall that
$$
\mathbb{P}_{B^n_{p, \bm{k}}(R), p, \bm{k}, \lambda}\left(|X|=t\right)=\frac{\lambda^t\hat{Z}_{B^n_{p, \bm{k}}(R), p, \bm{k}}(t)}{Z_{B^n_{p, \bm{k}}(R), p, \bm{k}}(\lambda)}.
$$
So
$$
\hat{Z}_{B^n_{p, \bm{k}}(R), p, \bm{k}}(t)\geq\frac{1}{V\lambda^t}Z_{B^n_{p, \bm{k}}(R), p, \bm{k}}(\lambda)\geq\frac{1}{V\lambda^t}.
$$

On the other hand, by Theorem \ref{main} and the definition of $\alpha_{B^n_{p, \bm{k}}(R), p, \bm{k}}(\lambda)$, we have
$$
\mathbb{E}_{B^n_{p, \bm{k}}(R), p, \bm{k}, \lambda}(|X|)=\alpha_{B^n_{p, \bm{k}}(R), p, \bm{k}}(\lambda)\cdot V\geq(1+o_n(1))\frac{\log(2/c_p)\cdot n}{2^n}\cdot V.
$$
Let $\epsilon$ be a small positive number (we will choose $\epsilon$ dependent on $n$ but independent of $V$) and $\alpha=(1+o_{n}(1))\frac{\log(2/c_p)\cdot n}{2^n}-\epsilon$.
So for large $V$,
$$
\alpha\leq(1+o_{n}(1))\frac{\log(2/c_p)\cdot n}{2^n}-V^{-1/5}\leq\frac tV.
$$
Applying Lemma \ref{dijian}, it follows that
$$
f_{p, \bm{k}, n}(\alpha)=\lim_{V\rightarrow\infty}\frac{1}{\alpha V}\log\frac{\hat{Z}_{B^n_{p, \bm{k}}(R), p, \bm{k}}(\lfloor\alpha V\rfloor)}{V^{\lfloor\alpha V\rfloor}/\lfloor\alpha V\rfloor!}\geq\lim_{V\rightarrow\infty}\frac{1}{t}\log\frac{\hat{Z}_{B^n_{p, \bm{k}}(R), p, \bm{k}}(t)}{V^{t}/t!}.
$$
We calculate
\begin{equation*}
\begin{split}
\lim_{V\rightarrow\infty}\frac{1}{t}\log\frac{\hat{Z}_{B^n_{p, \bm{k}}(R), p, \bm{k}}(t)}{V^{t}/t!}\geq&\lim_{V\rightarrow\infty}\frac{1}{t}\log\frac{t!}{V^{t+1}\lambda^t}\\
=&\lim_{V\rightarrow\infty}\frac{1}{t}\log\frac{t!}{V^{t+1}}-\log\lambda\\
\geq&\lim_{V\rightarrow\infty}\frac{1}{t}\log\frac{t!}{V^{t+1}}+n\log c_p-\log2.
\end{split}
\end{equation*}
Note that $t\rightarrow\infty$ when $V\rightarrow\infty$. Hence Stirling's formula implies that
\begin{equation*}
\begin{split}
\lim_{V\rightarrow\infty}\frac{1}{t}\log\frac{t!}{V^{t+1}}=&\lim_{V\rightarrow\infty}\frac{1}{t}\log\frac{\sqrt{2\pi t}(t/e)^t}{V^{t+1}}\\
=&\lim_{V\rightarrow\infty}\left(\frac{1}{t}\log\sqrt{2\pi t}+\log(t/V)-1-\frac{1}{t}\log(1/V)\right)\\
=&\lim_{V\rightarrow\infty}\log(t/V)-1\\
\geq&\log\left((1+o_{n}(1))\frac{\log(2/c_p)\cdot n}{2^n}\right)-1\\
=&(-\log2+o_n(1))n.
\end{split}
\end{equation*}
Therefore,
$$
f_{p, \bm{k}, n}(\alpha)\geq(-\log2+o_n(1))n+n\log c_p-\log2=-(1+o_n(1))\log(2/c_p)\cdot n.
$$
Choose $\epsilon=o_n\left(\frac{\log(2/c_p)\cdot n}{2^n}\right)$ so that $\alpha=(1+o_{n}(1))\frac{\log(2/c_p)\cdot n}{2^n}-\epsilon=(1+o_{n}(1))\frac{\log(2/c_p)\cdot n}{2^n}$, completing the proof.
\end{proof}

\section*{Acknowledgements}
The authors would like to thank Professor Yufei Zhao and Professor Chuanming Zong for many helpful and valuable comments on the manuscript. The authors also thank Professor Hong Liu for sharing his nice idea about the alternate proof.

\bibliographystyle{abbrv}
\bibliography{sphere_packing_REF}
\end{document}